\documentclass{amsart}
\usepackage[utf8]{inputenc}

\usepackage{latexsym}
\usepackage{amsthm}
\usepackage{amssymb,amsmath}
\usepackage[dvips]{graphicx}
\usepackage[dvips]{graphics}
\usepackage{graphicx}
\usepackage{subfigure}
\usepackage{color}
\usepackage{xcolor}
\usepackage{epstopdf}
\usepackage{hyperref}
\usepackage{multirow}
\usepackage{enumerate}
\usepackage{extarrows}
\usepackage{pgfplots}
\usepackage{tikz-3dplot}
\usepackage[bottom]{footmisc}

\tdplotsetmaincoords{60}{115}
\pgfplotsset{compat=newest}

\newtheorem{thm}{Theorem}[section]
\newtheorem{lem}[thm]{Lemma}
\newtheorem{prop}[thm]{Proposition}
\newtheorem{cor}{Corollary}
\newtheorem{claim}{Claim}

\theoremstyle{definition}
\newtheorem{defn}{Definition}
\newtheorem{conj}{Conjecture}
\newtheorem{eg}{Example}

\theoremstyle{remark}

\newcommand{\bC}{\mathbb C}
\newcommand{\bR}{\mathbb R}

\newcommand{\bP}{\mathbb P}
\newcommand{\bN}{\mathbb N}
\newcommand{\bE}{\mathbb E}

\newcommand{\cR}{\mathcal R}

\newcommand{\ZZ}{\mathbb Z}

\newcommand{\ul}{\underline}

\newcommand{\Sp}{S^{p-1}}
\newcommand{\Rnon}{\mathbb{R}_{\geq 0}}
\newcommand{\ener}{\mathcal E_{n,p}}

\newcommand{\Enp}{E_{n,p}}
\newcommand{\Unif}{\operatorname{Unif}}

\newcommand\pgfmathsinandcos[3]{%
	\pgfmathsetmacro#1{sin(#3)}%
	\pgfmathsetmacro#2{cos(#3)}%
}
\newcommand\LongitudePlane[3][current plane]{%
	\pgfmathsinandcos\sinEl\cosEl{#2} 
	\pgfmathsinandcos\sint\cost{#3} 
	\tikzset{#1/.style={cm={\cost,\sint*\sinEl,0,\cosEl,(0,0)}}}
}
\newcommand\LatitudePlane[3][current plane]{%
	\pgfmathsinandcos\sinEl\cosEl{#2} 
	\pgfmathsinandcos\sint\cost{#3} 
	\pgfmathsetmacro\yshift{\cosEl*\sint}
	\tikzset{#1/.style={cm={\cost,0,0,\cost*\sinEl,(0,\yshift)}}} %
}
\newcommand\DrawLongitudeCircle[2][1]{
	\LongitudePlane{\angEl}{#2}
	\tikzset{current plane/.prefix style={scale=#1}}
	\pgfmathsetmacro\angVis{atan(sin(#2)*cos(\angEl)/sin(\angEl))} %
	\draw[current plane] (\angVis:1) arc (\angVis:\angVis+180:1);
	\draw[current plane,dashed] (\angVis-180:1) arc (\angVis-180:\angVis:1);
}
\newcommand\DrawLatitudeCircle[2][2]{
	\LatitudePlane{\angEl}{#2}
	\tikzset{current plane/.prefix style={scale=#1}}
	\pgfmathsetmacro\sinVis{sin(#2)/cos(#2)*sin(\angEl)/cos(\angEl)}
	\pgfmathsetmacro\angVis{asin(min(1,max(\sinVis,-1)))}
	\draw[current plane] (\angVis:1) arc (\angVis:-\angVis-180:1);
	\draw[current plane,dashed] (180-\angVis:1) arc (180-\angVis:\angVis:1);
}

\DeclareMathOperator*{\argmax}{arg\,max}
\DeclareMathOperator*{\argmin}{arg\,min}
\DeclareMathOperator*{\spn}{span}

\numberwithin{equation}{section}

\begin{document}

\title{On the Minimax Spherical Designs}

\author{Weibo Fu}
\address{Department of Mathematics, Princeton University, Princeton, NJ  08544}
\email{wfu@math.princeton.edu}

\author{Guanyang Wang}
\address{Department of Statistics, Hill Center, Rutgers University, 110 Frelinghuysen Road, Piscataway, NJ 08854}
\email{guanyang.wang@rutgers.edu}

\author{Jun Yan}
\address{New York, NY 10036}
\email{yanjun9306@gmail.com}


\date{Weibo Fu is a Ph.D. student in Department of Mathematics, Princeton University, email: \href{wfu@math.princeton.edu}{wfu@math.princeton.edu}. Guanyang Wang is a tenure-track assistant professor in Department of Statistics, Rutgers University, email \href{guanyan.wang@rutgers.edu}{guanyan.wang@rutgers.edu}. Jun Yan has email: \href{yanjun9306@gmail.com}{yanjun9306@gmail.com}}.



\begin{abstract}
Distributing points on a (possibly high-dimensional) sphere  with minimal energy  is a long-standing problem in and outside the field of mathematics.
This paper considers a novel energy function that arises naturally from statistics and combinatorial optimization, and studies its theoretical properties.
Our result solves both the exact optimal spherical point configurations in certain cases and the minimal energy asymptotics under general assumptions. Connections between our results and the L1-Principal Component analysis and Quasi-Monte Carlo methods are also discussed. 
\end{abstract}

\maketitle

\section{Introduction}\label{sec: intro}
The problem of distributing  points  on a sphere with minimal energy has attracted much interest in various branches of science. Mathematically, let $p$ be a positive integer. We denote by $S^{p-1}= \left\{v \in \bR^p : \|v\| =1\right\}$  the unit sphere in $\bR^p$, where $\|\cdot\|$ stands for the standard Euclidean norm. For each positive integer $n$ and a predefined energy function $E_{n,p}: (\Sp)^n \rightarrow \Rnon$, we are interested in finding the minimal energy
\begin{equation}\label{eqn: minimal energy}
\ener := \inf_{\ul u  \in (\Sp)^n} E_{n,p}(\ul u)
\end{equation}
where $\ul u = (u_1, \cdots, u_n)$ is a set of $n$ points on the unit sphere, and the corresponding optimal configurations (i.e., minimizers of the above energy function)
\begin{equation}\label{eqn: optimal configuration}
\ul u^\star := \argmin_{\ul u\in (\Sp)^n} E_{n,p}(\ul u).
\end{equation}

The minimal energy, unsurprisingly, depends on the energy function $E_{n,p}$. Finding the minimal energy and the corresponding optimal configurations is a fundamental problem in extremal geometry.  In the existing literature, the energy function usually takes the form $\sum\limits_{i\neq j} f(\lVert u_i - u_j \rVert)$ where $f: \Rnon \rightarrow \Rnon$ is a decreasing function. For example, on the unit $2$-sphere ($p = 3$), the problem is known as the Smale's seventh problem \cite{smale1998mathematical} when $f(x) = 1/\log x$, the Thomson problem \cite{thomson1904xxiv} when $f(x) =1/x$, the generalized Thomson problem  or Riesz energy problem when $f(x) = 1/x^\alpha$ for some $\alpha> 0$.  Moreover, the problem is known as the Tammes problem \cite{tammes1930origin} if $E_{n,p}(\ul u) :=1/ \min\lVert u_i - u_j\rVert$. For the general $p$-sphere, the optimal configurations are naturally connected with the well-known spherical design problem \cite{delsarte1991spherical}. The mentioned problems are interconnected with each other, but also exciting fields independently, attracting many  researchers. Taking the Thomson problem as an example, the exact optimal configurations for $S^2$ have only been solved for  $n\leq 6$ and $n = 12$, where the case $n = 5$ is solved using a sophisticated computer-assisted proof \cite{schwartz2013five}. The asymptotics  for the minimal energy of the generalized Thomson problem under different regimes are derived in  \cite{wagner1990means} \cite{wagner1992means} \cite{kuijlaars1998asymptotics}.  We also refer the interested readers to  \cite{cohn2007universally}, \cite{katanforoush2003distributing} \cite{saff1997distributing} and the references therein for other  related results.

In this paper, we consider  a new energy function defined as
\begin{equation}\label{eqn: max energy}
\Enp(\ul u) := \max_{v\in S^{p-1}} \sum_{i=1}^n \lvert u_i \cdot v\rvert.
\end{equation}
The optimal configurations which minimize (\ref{eqn: max energy}) are called the \textit{minimax spherical designs} as it can be written as:
\begin{equation}\label{eqn: minimax design}
\ul u^\star := \argmin_{\ul u\in (\Sp)^n} \max_{v\in S^{p-1}} \sum_{i=1}^n \lvert u_i \cdot v\rvert.
\end{equation}

The minimax spherical design is also related to the traditional  $t$-designs for spheres \cite{delsarte1991spherical} and projective spaces \cite{hoggar1982t}. However, their math formulations  are different from our setup. In $t$-designs, the fixed parameter $t$ stands for the degree of polynomials. A $t$-design is a collection of points $X$ on the space of interest such that the integration over polynomials with degree no larger than $t$ matchs their averages on $X$. 	In our case, the fixed parameter $n$ is the number of points on the sphere, and we look for $n$ points that minimizes the energy function as described in \eqref{eqn: minimax design}.

One can observe that the new energy function (\ref{eqn: max energy}) is invariant under the elementwise-reflection over the origin, that is, $\Enp(u_1,\cdots,u_n) = \Enp(s_1u_1,\cdots, s_n u_n)$ where $(s_1,\cdots, s_n)$ is an arbitrary vector in $\{-1,1\}^n$. Therefore, minimizing  (\ref{eqn: max energy}) over $n$ vectors on the $p-1$-sphere is equivalent to minimizing (\ref{eqn: max energy}) over the upper hemisphere. Therefore the minimax design can also be viewed as a way of distributing points evenly on a hemisphere, or equivalently the real projective space $\mathbb{RP}^{p-1}$. As we will see  later, this new energy functional arises naturally and has applications in combinatorial optimization,  L1-Principal Component analysis (L1-PCA) and quasi-Monte Carlo.  Moreover, as we will see  in Lemma \ref{lem: discrete optimization}, finding the minimal energy \eqref{eqn: max energy} is equivalent to the combinatorial optimization problem \eqref{eqn:minimal energy, alternative form}. Formula \ref{eqn:minimal energy, alternative form} shares many similarities with the $L^2$ or spherical discrepancy \cite{alon2016probabilistic, chazelle2001discrepancy}, and therefore our techniques may be of independent interest.

In this paper we consider both the exact optimal configurations under certain circumstances and the asymptotics of the minimal energy (\ref{eqn: max energy}) under general assumptions. Our results are briefly summarized and discussed below:

\begin{itemize}
	\item We derive the  exact minimax spherical deisgns and the corresponding minimal energy in the following three cases:  
	\begin{enumerate}
		\item
\textit{Case 1: $ p \geq n$}, the minimax design is the set of $n$ mutually orthogonal vectors  with $\mathcal E_{n,p} = \sqrt{n}$.
\item \textit{Case 2: $p = 2, n$ arbitrary}, the minimax design is  the evenly spaced points on the upper semi-circle with $\mathcal{E}_{n,2} = \sin^{-1}(\frac{\pi}{2n})$ (see also Figure \ref{fig: 2d spherical design} for  illustration of the case $n = 5$).
\item \textit{Case 3: $p = 3, n = 4$}, the minimax design is of the  so-called triangular pyramid type (see Definition \ref{def: triangular pryamid}) with $\mathcal{E}_{4,3} = \sqrt {5}$. See also  Figure \ref{fig: 3d spherical design} for illustrations.
	\end{enumerate}
	We want to point out that Case 1 and Case 2  are essentially known (see, for example, \cite{yu2020optimal}) but under different notions.  Case 3  and has not been studied before which is more interesting and complicated.  As side products, we  characterized the local minimas of the energy function (\ref{eqn: max energy}) and therefore obtained all the local minimas when $p = 3, n =4$. For the sake of completeness, we give independent proofs for all the three cases.
	\item We  derive the asymptotics of the minimal energy $\ener$ defined in (\ref{eqn: minimal energy}) and construct the asymptotically minimax designs. To be more precise, we prove:
	\begin{enumerate}
		\item  When $p$ is  arbitrarily fixed and $n\rightarrow \infty$, we have
		\begin{equation}\label{eqn: asymptotics, i}
				\frac{\Gamma(p/2)}{\sqrt{\pi}   \Gamma((p+1)/2)}  \leq 	\frac{\ener}{n} \leq \frac{\Gamma(p/2)}{\sqrt{\pi}   \Gamma((p+1)/2)}  + C_pn^{-\frac 1p}.
		\end{equation}
		In other words, $\ener \sim  \frac{\Gamma(p/2)}{\sqrt{\pi}   \Gamma((p+1)/2)}  n$. 
		\item When $n,p$ are two arbitrary  positive integers with $n > p$, we have $\ener = \Theta(n/\sqrt p)$. More precisely:
			\begin{equation}\label{eqn: asymptotics, ii}
	\sqrt \frac{2}{\pi} \cdot\frac{n} {\sqrt {p+1}}< \ener < \frac{\sqrt 5}{2}\cdot \frac{n}{\sqrt p}.
		\end{equation}
	\end{enumerate}
	Moreover, we construct the configurations that attain the minimal energy asymptotically using the sphere's area-regular partitions. We further conjecture the quantity $\ener/n$ which represents the average energy  is decreasing with $n$ when $p$ is fixed, but we do not know how to prove it. 
\end{itemize}

\begin{center}
	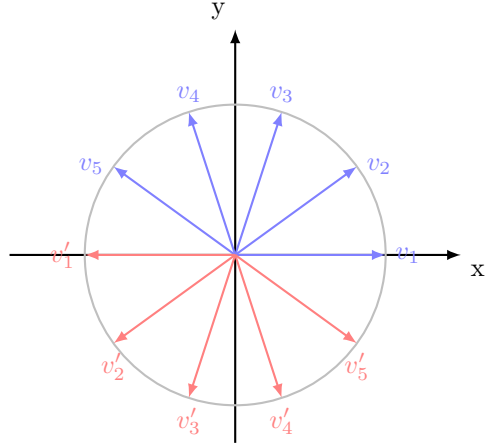
\begin{figure}[htbp]
		
		\begin{tikzpicture}[>=latex]
		
		\draw[thick,->] (-3,0) -- (3,0) node[anchor=north west] {x};
		\draw[thick,->] (0,-2.5) -- (0,3) node[anchor=south east] {y };
		\draw[blue!50, thick,->] (0,0) coordinate (O) -- (0:2) coordinate (oa) 
		node[right] {$v_1$};
		
		\draw[blue!50, thick,->] (0,0) coordinate (O) -- (36:2) coordinate (ob) 
		node[right] {$v_2$};
		
		\draw[blue!50, thick,->] (0,0) coordinate (O) -- (72:2) coordinate (oc) 
		node[above] {$v_3$};
		
		\draw[blue!50, thick,->] (0,0) coordinate (O) -- (108:2) coordinate (od) 
		node[above] {$v_4$};
		
		\draw[blue!50, thick,->] (0,0) coordinate (O) -- (144:2) coordinate (oe) 
		node[left] {$v_5$};

		\draw[red!50, thick,->] (0,0) coordinate (O) -- (180:2) coordinate (oa) 
		node[left] {$v_1'$};
		
		\draw[red!50, thick,->] (0,0) coordinate (O) -- (216:2) coordinate (ob) 
		node[below] {$v_2'$};
		
		\draw[red!50, thick,->] (0,0) coordinate (O) -- (252:2) coordinate (oc) 
		node[below] {$v_3'$};
		
		\draw[red!50, thick,->] (0,0) coordinate (O) -- (288:2) coordinate (od) 
		node[below] {$v_4'$};
		
		\draw[red!50, thick,->] (0,0) coordinate (O) -- (324:2) coordinate (oe) 
		node[below] {$v_5'$};
		
		\draw[gray!50, thick] (O) circle (2 cm);
		\end{tikzpicture}
		
		\caption{The spherical design for $n = 5, p = 2$. The five blue vectors $\{v_1, v_2, \cdots, v_5\}$ are five evenly spaced points on the semicircle, the rest five red vectors $\{v_1', v_2', \cdots, v_5'\}$ are the antipodal points of $\{v_1, v_2, \cdots, v_5\}$. } 
		\label{fig: 2d spherical design}
	\end{figure}
	
	\begin{figure}[htbp]
		
		\begin{tikzpicture}[tdplot_main_coords, scale = 2]
		
		\coordinate (P) at ({0},{0},{1});
		\coordinate (Q) at ({1},{0},{0});
		\coordinate (R) at ({-1/2},{-sqrt(3)/2},{0});
		\coordinate (S) at ({-1/2},{sqrt(3)/2},{0});
		\coordinate (P') at ({0},{0},{-1});
		\coordinate (Q') at ({-1},{0},{0});
		\coordinate (R') at ({1/2},{sqrt(3)/2},{0});
		\coordinate (S') at ({1/2},{-sqrt(3)/2},{0});
		\shade[ball color = lightgray,
		opacity = 0.5
		] (0,0,0) circle (1cm);
		
		\tdplotsetrotatedcoords{0}{0}{0};
		\draw[dashed,
		tdplot_rotated_coords,
		gray
		] (0,0,0) circle (1);
		
		\tdplotsetrotatedcoords{90}{90}{90};
		\draw[dashed,
		tdplot_rotated_coords,
		gray
		] (1,0,0) arc (0:180:1);
		
		\tdplotsetrotatedcoords{0}{90}{90};
		\draw[dashed,
		tdplot_rotated_coords,
		gray
		] (1,0,0) arc (0:180:1);
		
		
		\draw[-stealth] (0,0,0) -- (1.80,0,0);
		\draw[-stealth] (0,0,0) -- (0,1.30,0);
		\draw[-stealth] (0,0,0) -- (0,0,1.30);
		\draw[dashed, gray] (0,0,0) -- (-1,0,0);
		\draw[dashed, gray] (0,0,0) -- (0,-1,0);
		
		\draw[dashed, -stealth, blue!50] (0,0,0) -- (P);
		\draw[dashed, -stealth, blue!50] (0,0,0) -- (Q);
		\draw[dashed, -stealth, blue!50] (0,0,0) -- (R);
		\draw[dashed, -stealth, blue!50] (0,0,0) -- (S);
			\draw[thick , blue!50] (Q) -- (P);
		\draw[thick , blue!50] (Q) -- (R);
		\draw[thick , blue!50] (Q) -- (S);
		\draw[thick , blue!50] (P) -- (S);
		\draw[thick, blue!50] (P) -- (R);
		\draw[thick, blue!50] (R) -- (S);
		
		\draw[dashed, -stealth, red!50] (0,0,0) -- (P');
		\draw[dashed, -stealth, red!50] (0,0,0) -- (Q');
		\draw[dashed, -stealth, red!50] (0,0,0) -- (R');
		\draw[dashed, -stealth, red!50] (0,0,0) -- (S');
		
		\draw[fill = lightgray!50] (P) circle (0.5pt);
		\draw[fill = lightgray!50] (Q) circle (0.5pt);
		\draw[fill = lightgray!50] (R) circle (0.5pt);
		\draw[fill = lightgray!50] (S) circle (0.5pt);
		\draw[fill = lightgray!50] (P') circle (0.5pt);
		\draw[fill = lightgray!50] (Q') circle (0.5pt);
		\draw[fill = lightgray!50] (R') circle (0.5pt);
		\draw[fill = lightgray!50] (S') circle (0.5pt);
		\end{tikzpicture}
		
		\caption{The spherical design for $n = 3, p = 4$. The four blue vertex of the triangular pyramid forms a spherical design. The four red vectors are the antipodal points.}
		\label{fig: 3d spherical design}
	\end{figure}
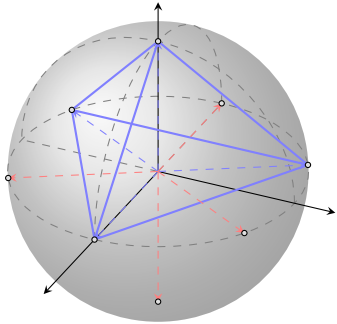	
	
\end{center}

Interestingly, the proof techniques for the exact minimax spherical designs and the asymptotics are quite different. Finding the exact minimax spherical designs relies on combinatorial methods.  For example,  the combinatorial trick given in Lemma \ref{lem: discrete optimization} reduces the problem of maximizing a function (\ref{eqn: max energy})
over a compact region into an issue of optimizing a function over a finite (but still exponentially large) set. 
Moreover, by allowing infinitesimal variations at the local minima, we are able to get extra incidence relations Lemma \ref{lem: local minimal}, and exploit them to better understand and analyze the general cases.
With additional combinatorial arguments, these incidence relations help us completely settle down the problem for $n=4, p=3$.
In contrast,  though the original problem itself is deterministic, the asymptotic results mostly rely on  probabilistic methods. The lower bound in (\ref{eqn: asymptotics, i}) is proved directly using probability arguments, and the upper bound combines  probabilistic arguments with results in area-regular partitions for a unit sphere.

The rest of this paper is organized as follows. Section \ref{sec: exact designs} solves the minimal energy and the corresponding minimax spherical designs for Case 1 - Case 3 mentioned above. Section \ref{sec: asymptotic energy} studies the asymptotic behaviors of the minimal energy $\ener$, and construct the asymptotically minimax designs. Section \ref{sec: application} discusses two applications: L1-PCA and Quasi-Monte Carlo. Several unsolved problems are discussed at the end of each section.

\section*{Acknowledgement}

The authors would like to thank Persi Diaconis, Andrea Ottolini, and Xiaoming Huo for helpful comments and discussions.

\section{Exact minimax  designs}\label{sec: exact designs}

This section focuses on solving the minimal energy and the corresponding minimax designs in the three cases mentioned in Section \ref{sec: intro}. We start with proving Lemma \ref{lem: discrete optimization} which will be useful throughout this section. Case 1 can also be viewed as an application of  Lemma \ref{lem: discrete optimization}, and is proved in Proposition \ref{prop:case1}. Case 2 and Case 3 are  proved in Proposition \ref{prop:case2} and \ref{prop: case3} separately. In particular, Case 3 is technically most complicated, and the proof relies on repeatedly using the idea and result of the key lemma -- Lemma \ref{lem: local minimal}.

For a fix set of $n$ points $\ul u  = (u_1, \cdots, u_n)$ on the unit $p$-sphere $\Sp$, finding the energy $\Enp(\ul u)$  is equivalent  to maxizing the function $g_{\ul u}(v): = \sum_{i=1}^n  \lvert u_i \cdot v\rvert$ over $\Sp$. The following lemma (also proved in \cite{borodachov2008asymptotics}) shows that the above problem is equivalent to a discrete combinatorial optimization problem.

\begin{lem}\label{lem: discrete optimization}
	With the energy function $E_{n,p}$ defined as in (\ref{eqn: max energy}), for each $\ul u  = (u_1, \cdots, u_n)\in (\Sp)^n$, we have:
	\begin{equation}\label{eqn: discrete optimization}
	\Enp(\ul u)  = \max_{v\in S^{p-1}} \sum_{i=1}^n \lvert u_i \cdot v\rvert = \max_{\delta \in \{-1,1\}^n} \lVert \sum_{i=1}^n \delta_i u_i\rVert.
	\end{equation}
\end{lem}

\begin{proof}
	\begin{align*}
	\Enp(\ul u)  &= \max_{v\in S^{p-1}} \sum_{i=1}^n \lvert u_i \cdot v\rvert =  \max_{v\in S^{p-1}} \sum_{i=1}^n \max_{\delta_i\in\{-1,1\} }   u_i \cdot (\delta_i v) \\
	& = \max_{v\in S^{p-1}} \max_{\delta \in \{-1,1\}^n} \big(\sum_{i=1}^n (\delta_i u_i) \cdot  v \big)\\
	& = \max_{\delta \in \{-1,1\}^n} \max_{v\in S^{p-1}} \big(\sum_{i=1}^n (\delta_i u_i) \cdot  v  \big) \\
	& = \max_{\delta \in \{-1,1\}^n} \lVert \sum_{i=1}^n \delta_i u_i\rVert,
	\end{align*}
	where the last equality follows from the Cauchy-Schwarz inequality. After fixing $\delta \in \{-1,1\}^n$, it is straightforward from Cauchy-Schwarz inequality that the quantity $\big(\sum_{i=1}^n (\delta_i u_i) \cdot  v  \big)$ is no larger than  $\lVert \sum_{i=1}^n \delta_i u_i\rVert \lVert v\rVert =  \lVert \sum_{i=1}^n \delta_i u_i\rVert $ and is maximized by taking 
	\[
	v = \frac{\sum_{i=1}^n \delta_i u_i}{\lVert \sum_{i=1}^n \delta_i u_i\rVert }.
	\]
\end{proof}

Lemma \ref{lem: discrete optimization} suggests, instead of searching for all the points on the sphere $\Sp$, it suffices to maximize the  vectors' Euclidean norm over a finite set of the vectors which are of the form $\sum_{i=1}^n \delta_i u_i$. Consequently, the minimal energy can be equivalently written as:
\begin{equation}\label{eqn:minimal energy, alternative form}
\ener = \min_{\substack{\ul u \in (\Sp)^n }} \max_{\substack{\delta \in \{-1, 1\}^n }} \| \sum_{i=1}^n \delta_i u_i \|.
\end{equation}

The problem for $p=1$ is trivial. For $n \leq p$, it is easy to prove the above minimal energy (\ref{eqn:minimal energy, alternative form}) equals  $\sqrt {n}$, where the equality holds if and only if all the $u_i$ are orthogonal to each other.
\begin{prop}[Case 1: $n\leq p$]\label{prop:case1}
	For any  positive integers $n,p$, we have $\ener \geq \sqrt n$, where the equality holds if and only if $n\leq p$ and  all the $u_i$ are orthogonal to each other.
\end{prop}

\begin{proof}
Consider the following equality
$$\sum_{\substack{\delta \in \{\pm 1\}^n}} \| \sum_{i=1}^n \delta_i v_i \|^2 = n 2^n,$$
which is true for every $\{v_1, \cdots, v_n\}$ as we can expand the expression and cancel out all the cross-terms. 
We immediately have $n2^n \leq 2^n \max_{\delta \in \{\pm 1\}^n}  \| \sum_{i=1}^n \delta_i v_i \|^2$. Taking the mininum over $\{v_1, \cdots, v_n\}$ on the RHS yields $\ener \geq \sqrt n$.
The equality holds if and only if $\| \sum_{i=1}^n \delta_i v_i \|^2 = n$ for any $\delta_i \in \{\pm 1\}, ~ 1 \leq i \leq n$. Therefore, $$\| v_n + \sum_{i=1}^{n-1} \delta_i v_i \|^2 = \| - v_n + \sum_{i=1}^{n-1} \delta_i v_i \|^2 = n,$$ implying $v_n \perp \sum_{i=1}^{n-1} \delta_i v_i$,
which shows $v_n \perp v_i$ for all $1 \leq i \leq n-1$. We win by an easy induction.
\end{proof}

\subsection{Case 2: $p=2$, $n$ arbitrary}\label{subsec: case 2}
We now turn to Case 2. The minimal energy result depends on the following lemma. 
\begin{lem}\label{lem: cos global max}
	Let $k$ be a positive integer. 
	Fix  $\theta \geq 0$, let $A_k^\theta$ be the bounded region 
	$$A_k^\theta:=\left\{ (\theta_1,\cdots,\theta_k) \in [0,\pi]^k | \sum_{i=1}^k \theta_i = \theta, ~ \theta_i + \theta_{i+1} \leq \pi, ~ 1 \leq i \leq k -1 ~\mathrm{and}~ \theta_k+\theta_1 \leq \pi \right\}.$$
	We define a function $C$ on $A_k^\theta$
	\begin{eqnarray*}
		C : A_k^\theta & \to & \bR_{\geq 0} \\
		\ul{\theta}=(\theta_1,\cdots,\theta_k) & \mapsto & \sum\limits_{i=1}^k \cos(\theta_i).
	\end{eqnarray*}
	Suppose $A_k^\theta$ is non-empty, or equivalently $\theta \leq \frac{k\pi}{2}$, then $C$ attains its maximal value at $\ul{\theta}=(\frac{\theta}{k},\cdots,\frac{\theta}{k})$ with $C(\frac{\theta}{k},\cdots,\frac{\theta}{k})=k\cos(\frac{\theta}{k})$. 
	In particular, if $\theta < \frac{k\pi}{2}$, then the only maximal value is attained at $\ul{\theta}=(\frac{\theta}{k},\cdots,\frac{\theta}{k})$. 
\end{lem}

\begin{proof}
	For easing notations, we regard indices of $\theta$ as elements in $\ZZ / k \ZZ$. In particular, $\theta_{k+1}=\theta_1$.
	When $\theta = \frac{k\pi}{2}$ and $k$ is odd,  the definition of $A_k^\theta$ forces $\theta_1 = \theta_2 = \cdots = \theta_k = \frac \pi 2$. When $\theta = \frac{k\pi}{2}$ and $k$ is even, 
	it is clear that we can take $\theta_1 = \theta_2 = \cdots =\theta_k = \frac{\pi}{2}$ which attains the maximum $C(\ul \theta)= 0$. 
	
	Now we assume $\theta < \frac{k\pi}{2}$.
	Since $A_k^\theta$ is compact, function $C$ attains its maximal at some $\ul{\theta}=(\theta_1,\cdots,\theta_k)\in A_k^\theta$.
	We claim that if $\ul{\theta} \neq (\frac{\theta}{k},\cdots,\frac{\theta}{k})$, there must exist $\theta_i$ such that $\theta_i > \theta_{i+1}$ and $\theta_{i+1} + \theta_{i+2} < \pi$. 
	To see this, pick a $\theta_i$ such that $\theta_i = \max\{\theta_1,\cdots,\theta_k\}$, $\theta_i > \theta_{i+1}$.
	If the claim is not true, then $\theta_{i+1} + \theta_{i+2} = \pi$.
	$\theta_{i+2}=\pi-\theta_{i+1} \geq \theta_i$, therefore $\theta_{i+2}=\max\{\theta_1,\cdots,\theta_k\}=\theta_i$. 
	Now replace $\theta_i$ with $\theta_{i+2}$, we get $\theta_{i+2}=\theta_{i+4}$. 
	Hence $\theta_i=\theta_{i+2p}$ for all $p \in \ZZ$.
	If $k$ is odd, then all $\theta_i$'s are equal.
	If $k$ is even, then $\ul{\theta}=(\theta_1,\pi-\theta_1,\cdots,\theta_1,\pi-\theta_1)$, $\theta=\frac{k\pi}{2}$.
	Both are against the assumption. Therefore, the claim is proved.
	
	After choosing the aforementioned index $i$, we 
	 pick a small angle $\varepsilon > 0$ such that $\theta_i - \varepsilon > \theta_{i+1} + \varepsilon$ and $(\theta_{i+1}+\varepsilon) + \theta_{i+2} < \pi$.
	 Consider a new vector $\ul{\theta'}$ which has the $i$-index $\theta_i -\varepsilon$, $i+1$-th index $\theta_{i+1}+\varepsilon$ and equals $\ul \theta$ elsewhere. Straghtforward calculation gives
	$C(\ul{\theta}')-C(\ul{\theta})=\cos(\theta_i-\varepsilon)+\cos(\theta_{i+1}+\varepsilon)-\cos(\theta_i)-\cos(\theta_{i+1}) > 0$, which contradicts the maximal assumption. Therefore, the only maximal value of $C$ is attained at $(\frac{\theta}{k},\cdots,\frac{\theta}{k})$, as desired.

\end{proof}

Now we are ready to solve the $p = 2$ case. The minimax design is the evenly spaced points on the upper semi-circle or the evenly spaced points on the unit circle after adding all the antipodal points. 

\begin{prop}(Case 2: $p = 2$, n arbitrary)\label{prop:case2}
	$$
	\mathcal E_{n,2}=\min_{\substack{v_i \in S^1 \\ 1 \leq i \leq n}} \max_{\substack{\delta \in \{\pm 1\}^n}} \| \sum_{i=1}^n \delta_i v_i \| = \sin^{-1}(\frac{\pi}{2n}).$$
	Under the standard identification $\bC \simeq \bR^2$, $E_{n,2}(v_1, \cdots, v_n) = \mathcal E_{n,2}$  if and only if $\{\pm v_1,\cdots,\pm v_n\}$ is obtained from a rotation (by a group element of $\mathrm{SO}_2(\bR)$) of $\{1,e^{\frac{\pi i}{n}},e^{\frac{2 \pi i}{n}},\cdots,e^{(2n-1)\frac{\pi i}{n}}\}$.
\end{prop}

\begin{proof}
	Consider the unordered set $\{\pm v_1,\cdots,\pm v_n\} \subset S^1$, we reorder it such that $v_1,v_2,\cdots,v_n,v_{n+1}:=-v_1,v_{n+2}:=-v_2,\cdots,v_{2n}:=-v_m$ is of the anti-clockwise order.
	Let $S_l:=\sum\limits_{i=l}^n v_i - \sum\limits_{j=1}^{l-1} v_i$ for $1 \leq l \leq n$ and $S_l:=-S_{l-n}$ for $n+1 \leq l \leq 2n$.
	Let $S:=2 \sum\limits_{l=1}^n \| S_l \|^2=\sum\limits_{l=1}^{2n} \| S_l \|^2$.
	
	For convenience, we regard all the indices in $\ZZ/2n\ZZ$.
	For each $1 \leq k \leq n-1$, we use $\langle k \rangle$ to denote the cyclic group generated by $k$  in $\ZZ/2n\ZZ$. 
	It is  known that the greatest common divisor  $\mathrm{gcd}(k,2n) = \frac{2n}{|\langle k \rangle|}$, 
	and there are $\mathrm{gcd}(k,2n)$ cosets $C^k_1,\cdots C^k_{\mathrm{gcd}(k,2n)}$ in $\ZZ/2n\ZZ$ for $\langle k \rangle$.
	Let $S^k:=\sum\limits_{j=1}^{2n} v_j \cdot v_{j+k}=\sum\limits_{i=1}^{\mathrm{gcd}(k,2n)} \sum\limits_{j \in C^k_i} v_j \cdot v_{j+k}$.
	Notice that $v_j+v_{j+n}=0$, \begin{equation}\label{+=0}
	S^k+S^{n-k}=\sum\limits_{j=1}^{2n} v_j \cdot v_{j+k} + v_{j+k} \cdot v_{j+n}=0.
	\end{equation}
	
	For each $1 \leq k \leq \lfloor \frac{n}{2} \rfloor$, since
	the angle between $v_i$ and $v_{i+2k}$ is at most $\pi$, we have $\arccos(v_i \cdot v_{i+k}) + \arccos(v_{i+k} \cdot v_{i+2k}) \leq \pi$. Moreover,
	$$\sum\limits_{i=1}^{\mathrm{gcd}(k,2n)} \sum\limits_{j \in C^k_i} \arccos(v_j \cdot v_{j+k})=\sum\limits_{i=1}^{\mathrm{gcd}(k,2n)} \frac{2k\pi}{\mathrm{gcd}(k,2n)}=2k\pi.$$
	We can therefore apply  Lemma \ref{lem: cos global max}, which shows $S^k \geq 2n \cos(\frac{k\pi}{n})$.
	
	Now we calculate $S$ 
	\begin{eqnarray*}
		S & = & 2n^2 + 2 \sum_{k=1}^{n-1} (n-k) S^k \\
		& \stackrel{(\ref{+=0})}{=} & 2n^2 + 2 \sum_{k=1}^{\lfloor \frac{n}{2} \rfloor} (n-2k)S^k \\
		& \geq & 2n (n + 2 \sum_{k=1}^{\lfloor \frac{n}{2} \rfloor} (n-2k) \cos(\frac{k\pi}{n})) \\
		& = & \frac{4n}{1-\cos(\frac{\pi}{n})}.
	\end{eqnarray*}
	This implies $\| S_l \|^2 \geq \frac{2}{1-\cos(\frac{\pi}{n})}$ for some $1 \leq l \leq n$, and therefore 
	\begin{equation}\label{eqn:energy, p=2}
		E_{n,2}(v_1, \cdots, v_n)^2 = \max_{\substack{\delta\in \{\pm 1\}^n \\ 1 \leq i \leq n}} \| \sum_{i=1}^n \delta_i v_i \|^2 \geq \frac{2}{1-\cos(\frac{\pi}{n})} = \sin^{-2}(\frac{\pi}{2n}).
	\end{equation}
Since (\ref{eqn:energy, p=2}) holds for any $v_1, \cdots v_n$, we prove $\mathcal E_{n,2}^2 \geq  \sin^{-2}(\frac{\pi}{2n}).$ Moreover, the equality holds if and only if the global maximum in Lemma \ref{lem: cos global max} is attained, which means $\{\pm v_1,\cdots, \pm v_m\}$ is evenly distributed on $S^1$. Therefore $\{v_1, \cdots v_n\}$ is a minimax design for $p=2$ if and only if $\{\pm v_1,\cdots,\pm v_n\}$ is obtained from a rotation (by a group element of $\mathrm{SO}_2(\bR)$) of $\{1,e^{\frac{\pi i}{n}},e^{\frac{2 \pi i}{n}},\cdots,e^{(2n-1)\frac{\pi i}{n}}\}$.
\end{proof}

\subsection{Case 3: $p=3$, $n=4$}\label{subsec: case 3}
We now turn our attention to the  $p = 3, n > 3$ case. It turns out that finding the global minimal of  $E_{n,3}$ is quite difficult, as the function usually has more than one local minimal. Now we can only find all the local minimals and thus solve the case  $p = 3, n = 4$. We need  a few more definitions and  lemmas to study the properties of  local minimas.

Consider the function $$l:(\Sp)^n \to \Rnon, ~~ l(v_1,\cdots,v_n):=\max_{\substack{\delta \in \{-1, 1\}^n }} \| \sum_{i=1}^n \delta_i v_i \|^2.$$
We say $l$ attains its local minimal at $\ul{v}=(v_1,\cdots,v_n)$ if $l(\ul{v})$ is the minimal value of $l$  in a neighborhood of $\ul{v}$ in $(S^{p-1})^n$. 

For each fixed  $\ul{v} \in (\Sp)^n$ and each $\alpha=(\alpha_1,\cdots,\alpha_n) \in \{\pm 1\}^n$, we define $V_\alpha:=\sum\limits_{i=1}^n \alpha_i v_i$.
We also denote by $\pm M_{\ul{v}}:=   \{\beta \in\{\pm 1\}^n | \| V_\beta \|^2=l(\ul{v}) \}$ the index set which contains all the binary antipodal combinations of $v$ that attains $l(\ul v)$. It is clear that $\pm M_{\ul{v}}$ is invariant under sign flips, that is, $\alpha \in  \pm M_{\ul{v}}$ is equivalent to $-\alpha \in  \pm M_{\ul{v}}$.
Therefore we can choose $M_{\ul{v}} \subset \pm M_{\ul{v}}$ such that $\pm M_{\ul{v}}=M_{\ul{v}} \cup -M_{\ul{v}}$, and $M_{\ul{v}} \cap -M_{\ul{v}} = \emptyset$, where $ -M_{\ul{v}} := \{\alpha \in \{\pm 1\}^n | -\alpha \in  M_{\ul{v}}\}$. The next lemma studies the behavior of the local minimals of the function $l$.

\begin{lem}\label{lem: local minimal}
	Suppose $l$ attains its local minimal at $\ul{v}=(v_1,\cdots,v_n)$.
	For each $1 \leq i \leq n$ and $\alpha \in M_{\ul{v}}$, there exist $c^i_{\alpha} \in \bR$ such that the vector $(c^i_\alpha)_\alpha\neq \vec{0} \in \bR^{|M_{\ul{v}}|} $, and $\sum\limits_{\alpha \in M_{\ul{v}}} c^i_\alpha V_\alpha \in \bR^p$ is a scalar multiple of $v_i$ for every $1 \leq i \leq n$.
\end{lem}

\begin{proof}
	Let $T_i$ be the tangent space of $v_i$ on $\Sp$ translated to the orgin as a linear subspace of $\bR^p$ (comprising by vectors orthogonal to $v_i$).
	We define a linear map $D$ (which can be viewed as essentially a directional derivative) from $\prod\limits_{i=1}^n T_i$ to $\bR^{|M_{\ul{v}}|}$ as follows.
	\begin{eqnarray*}
		D : \prod\limits_{i=1}^n T_i & \to & \bR^{|M_{\ul{v}}|} \\
		(t_i)_{1\leq i\leq n} & \mapsto & (\langle V_\alpha , \sum\limits_{j=1}^n \alpha_j t_j \rangle)_{\alpha \in M_{\ul{v}}}
	\end{eqnarray*}

	By the minimal assumption, we claim:
	\begin{claim}\label{claim: intersection} The image of $D$ does not intersect with $\bR_{< 0}^{|M_{\ul{v}}|}$, in other words, every vector in the image of $D$ must have at least one non-negative coordinate. 
	\end{claim}
	Assume for the claim is true, then $D$ is clearly not surjective. Since $D$ is a linear but not surjective map,  there exists a nonzero vector which is orthogonal to the image of $D$. In other words, we can find a non-zero vector $(c_\alpha)_\alpha \in \bR^{|M_{\ul{v}}|}$ such that
	\begin{equation}\label{eqn: non-surjective map}
	\sum_{\alpha \in M_{\ul{v}}} c_\alpha \langle V_\alpha , \sum\limits_{j=1}^n \alpha_j t_j \rangle = 0, ~ \mathrm{for~ any~ } (t_i)_{1\leq i\leq n} \text{ with } t_i \in T_i ,1\leq i\leq n.
	\end{equation}
	Taking $c_\alpha^i:=c_\alpha \alpha_i$, as \[ 0 = \sum_{\alpha \in M_{\ul{v}}} c_\alpha \langle V_\alpha, \sum\limits_{j=1}^n \alpha_j t_j \rangle = \sum\limits_{i=1}^n \langle \sum\limits_{\alpha \in M_{\ul{v}}} c_\alpha^i V_\alpha, t_i \rangle, \] we conclude that $\langle \sum\limits_{\alpha \in M_{\ul{v}}} c_\alpha^i V_\alpha, t_i \rangle=0$ for any $t_i \in T_i$, therefore the vector $\sum\limits_{\alpha \in M_{\ul{v}}} c_\alpha^i V_\alpha$ is orthogonal to the tanget space $T_i$ and is in turn a scalar multiple of $v_i$, as desired.\end{proof}

We conclude Lemma \ref{lem: local minimal} by proving Claim \ref{claim: intersection}:
\begin{proof}[Proof of Claim \ref{claim: intersection}]
	Assume for contradiction that there exists a vector  $t^\circ := (t^\circ_i)_{i \in \{1,2,\cdots, n\}} \in  \prod\limits_{i=1}^n T_i$ such that $\langle V_\alpha , \sum\limits_{j=1}^n \alpha_j t^\circ_j \rangle < 0$ for every $\alpha \in M_{\ul{v}}$. We may  assume without loss of generality that $||t^\circ_i|| \leq 1$ for all $1 \leq i \leq m$. We can then pertube each $v_i$ a little bit to construct a new set of vectors $\ul{\tilde v}$ which has a smaller value of $l$, and therefore contradicts with the assumption that $l$ attains local minimum at $\ul v$. 
	
	Let $r := \max_{\alpha \in M_{\ul{v}}} \langle V_\alpha , \sum\limits_{j=1}^n \alpha_j t^\circ_j \rangle < 0$ and $\Delta = l(\ul v) - \max_{\alpha \notin \pm M_{\ul v}} \lVert V_\alpha \rVert > 0$. Choose a small $\epsilon > 0$ which satisfies
	\begin{equation}\label{eqn:inequality for epsilon, 1}
	(2n l(\ul v) + n^2) \varepsilon + 2n^2\varepsilon^2 + n^2\varepsilon^3 < -2r,
	\end{equation}
	and
	\begin{equation}\label{eqn:inequality for epsilon, 2}
	\epsilon <  \frac{\Delta}{4n}.
	\end{equation}

	Set $\ul{\tilde v} := (\tilde v_1, \cdots, \tilde v_n)$ where  $\tilde v_i : = \frac{v_i + \varepsilon t^\circ_i}{||v_i + \varepsilon t^\circ_i||}  $. For every $\alpha \in M_{\ul v}$, we calculate $\lVert \sum_{i=1}^n \alpha_i \tilde  v_i\rVert^2$ as:
	\begin{align*}
	\lVert \sum_{i=1}^n \alpha_i \tilde  v_i\rVert^2  =
	\lVert \sum_{i=1}^n \alpha_i \frac{v_i + \varepsilon t^\circ_i}{\lVert v_i + \varepsilon t^\circ_i\rVert} \rVert^2 & = \lVert \sum_{i=1}^n \bigg(\alpha_i \big( \frac{1}{\lVert v_i + \varepsilon t^\circ_i\rVert} -1 \big) (v_i + \varepsilon t_i^\circ)\bigg) + V_\alpha + \varepsilon \sum_{i=1}^n \alpha_i t_i^\circ \rVert ^2\\
	& \leq \bigg( \sum_{i=1}^n (\lVert v_i + \varepsilon t_i^\circ \rVert -1) + \lVert V_\alpha + \varepsilon \sum_{i=1}^n \alpha_i t_i^\circ\rVert\bigg)^2\\
	& \leq \bigg( n \varepsilon^2 + \lVert V_\alpha + \varepsilon \sum_{i=1}^n \alpha_i t_i^\circ\rVert \bigg) ^2 \\
	& = n^2\varepsilon^4 + \lVert V_\alpha + \varepsilon \sum_{i=1}^n \alpha_i t_i^\circ\rVert ^2 + 2n \varepsilon^2 \lVert V_\alpha + \varepsilon \sum_{i=1}^n \alpha_i t_i^\circ\rVert \\
	& \leq  \lVert V_\alpha \rVert^2 + n^2\varepsilon^2 + 2r\epsilon +  n^2\varepsilon^4 + 2n l(\ul v)\varepsilon^2 + 2n^2\varepsilon^3 <   \lVert V_\alpha \rVert^2
	\end{align*}
	where the first and second inequality are triangle inequalities and the last inequality is immediate after applying inequality \ref{eqn:inequality for epsilon, 1}. Meanwhile, for every $(\alpha_1, \cdots, \alpha_n) \in \{\pm 1\}^n$, the norm difference between $\sum_{i=1}^n \alpha_i v_i$ and the perturbed vector $\sum_{i=1}^n \alpha_i \tilde v_i$ can be bounded by:
	\begin{align*}
	\lVert \sum_{i=1}^n \alpha_i (  v_i - \tilde v_i) \ \rVert &\leq \sum_{i = 1}^n \lVert  v_i - \tilde v_i \rVert = \sum_{i = 1}^n \lVert  v_i - (v_i + \varepsilon t_i^\circ) + (1 - \frac{1}{\sqrt{1+\varepsilon^2}}) (v_i + \varepsilon t_i^\circ)\rVert\\
	& \leq \sum_{i=1}^n \bigg(\varepsilon + \sqrt {1+ \varepsilon^2} - 1\bigg) \leq 2n\epsilon < \frac{\Delta}{2}.
	\end{align*}
	Therefore, for every $\alpha \notin \pm M_{\ul{v}}$ we have, 
	\[
	\lVert \sum_{i=1}^n \alpha_i \tilde v_i  \rVert <  \lVert \sum_{i=1}^n \alpha_i  v_i  \rVert + \frac{\Delta}{2} \leq l(\ul v) - \frac{\Delta}{2},
	\]
	for every $\alpha \in \pm M_{\ul{v}}$, we have,
	\[
	\lVert \sum_{i=1}^n \alpha_i \tilde v_i  \rVert < l(\ul v).
	\]
	Combining the two cases above, we have $l(\ul {\tilde v}) = \max_{\alpha\in \{\pm 1\}^n} \lVert \sum_{i=1}^n \alpha_i \tilde v_i  \rVert < l(\ul v)$, which contradicts with the local minimal assumption.
\end{proof}
We single out two types of configurations when $n=3, ~ p=4$.

For $\ul{v}=(v_1,v_2,v_3,v_4), ~ v_i \in \bR^3$ up to permutations of $\{\pm v_1, \pm v_2, \pm v_3, \pm v_4\}$ and rotations of $\bR^3$ (under action of $\mathrm{O}(3)$), we define the cube type and the triangular pyramid type as follows (see also \ref{fig: two types} for illustrations):
\begin{defn}[Cube Type]\label{def: cube}
The set of vectors $\ul{v}=(v_1,v_2,v_3,v_4) \in (S^2)^4$ is defined to be of the cube type if $\{\pm v_1, \pm v_2, \pm v_3, \pm v_4\}$ are $8$ vertices of the inscribed cube inside $S^2$
\end{defn}
\begin{defn}[Triangular Pyramid Type]\label{def: triangular pryamid}
The set of vectors $\ul{v}=(v_1,v_2,v_3,v_4) \in (S^2)^4$ is defined to be of the triangular pyramid type if $\pm v_1$ are north and south poles and $v_2, v_3, v_4$ are vertices of an equilateral triangle on the equator.
\end{defn}

\begin{figure}[htbp]  
	\centering 
		\begin{tikzpicture}  [tdplot_main_coords, scale = 2]
	\coordinate (P) at ({sqrt(3)/3},{sqrt(3)/3},{sqrt(3)/3});
	\coordinate (Q) at ({sqrt(3)/3},{-sqrt(3)/3},{sqrt(3)/3});
	\coordinate (R) at ({-sqrt(3)/3},{sqrt(3)/3},{sqrt(3)/3});
	\coordinate (S) at ({-sqrt(3)/3},{-sqrt(3)/3},{sqrt(3)/3});
	\coordinate (P') at ({-sqrt(3)/3},{-sqrt(3)/3},{-sqrt(3)/3});
	\coordinate (Q') at ({-sqrt(3)/3},{sqrt(3)/3},{-sqrt(3)/3});
	\coordinate (R') at ({sqrt(3)/3},{-sqrt(3)/3},{-sqrt(3)/3});
	\coordinate (S') at ({sqrt(3)/3},{sqrt(3)/3},{-sqrt(3)/3});
	\shade[ball color = lightgray,
	opacity = 0.5
	] (0,0,0) circle (1cm);
	
	\tdplotsetrotatedcoords{0}{0}{0};
	\draw[dashed,
	tdplot_rotated_coords,
	gray
	] (0,0,0) circle (1);
	
	\tdplotsetrotatedcoords{90}{90}{90};
	\draw[dashed,
	tdplot_rotated_coords,
	gray
	] (1,0,0) arc (0:180:1);
	
	\tdplotsetrotatedcoords{0}{90}{90};
	\draw[dashed,
	tdplot_rotated_coords,
	gray
	] (1,0,0) arc (0:180:1);
	
	
	\draw[-stealth] (0,0,0) -- (1.80,0,0);
	\draw[-stealth] (0,0,0) -- (0,1.30,0);
	\draw[-stealth] (0,0,0) -- (0,0,1.30);
	\draw[dashed, gray] (0,0,0) -- (-1,0,0);
	\draw[dashed, gray] (0,0,0) -- (0,-1,0);
	
	\draw[dashed, blue!50] (0,0,0) -- (P) node[above] {$v_1$};
	\draw[dashed,  blue!50] (0,0,0) -- (Q) node[above] {$v_2$};
	\draw[dashed, blue!50] (0,0,0) -- (R) node[above] {$v_4$};
	\draw[dashed,  blue!50] (0,0,0) -- (S) node[above] {$v_3$};;
	\draw[thick, blue!50] (Q) -- (P);
	
	\draw[thick,  blue!50] (Q) -- (S);
	
	\draw[thick,  blue!50] (P) -- (R);
	\draw[thick,  blue!50] (R) -- (S);
	
	\draw[thick,red!50] (Q') -- (P');
	
	\draw[thick, red!50] (Q') -- (S');
	
	\draw[thick,red!50] (P') -- (R');
	\draw[thick, red!50] (R') -- (S');

	\draw[thick, black!50] (P) -- (S');
	
	\draw[thick, black!50] (S) -- (P');
	
	\draw[thick, black!50] (R) -- (Q');
	\draw[thick, black!50] (Q) -- (R');
	
	\draw[dashed, -stealth, red!50] (0,0,0) -- (P');
	\draw[dashed, -stealth, red!50] (0,0,0) -- (Q');
	\draw[dashed, -stealth, red!50] (0,0,0) -- (R');
	\draw[dashed, -stealth, red!50] (0,0,0) -- (S');
	
	\draw[fill = lightgray!50] (P) circle (0.5pt);
	\draw[fill = lightgray!50] (Q) circle (0.5pt);
	\draw[fill = lightgray!50] (R) circle (0.5pt);
	\draw[fill = lightgray!50] (S) circle (0.5pt);
	\draw[fill = lightgray!50] (P') circle (0.5pt);
	\draw[fill = lightgray!50] (Q') circle (0.5pt);
	\draw[fill = lightgray!50] (R') circle (0.5pt);
	\draw[fill = lightgray!50] (S') circle (0.5pt);  
	\end{tikzpicture} 
	\begin{tikzpicture} [tdplot_main_coords, scale = 2]
	\coordinate (P) at ({0},{0},{1});
	\coordinate (Q) at ({1},{0},{0});
	\coordinate (R) at ({-1/2},{-sqrt(3)/2},{0});
	\coordinate (S) at ({-1/2},{sqrt(3)/2},{0});
	\coordinate (P') at ({0},{0},{-1});
	\coordinate (Q') at ({-1},{0},{0});
	\coordinate (R') at ({1/2},{sqrt(3)/2},{0});
	\coordinate (S') at ({1/2},{-sqrt(3)/2},{0});
	\shade[ball color = lightgray,
	opacity = 0.5
	] (0,0,0) circle (1cm);
	
	\tdplotsetrotatedcoords{0}{0}{0};
	\draw[dashed,
	tdplot_rotated_coords,
	gray
	] (0,0,0) circle (1);
	
	\tdplotsetrotatedcoords{90}{90}{90};
	\draw[dashed,
	tdplot_rotated_coords,
	gray
	] (1,0,0) arc (0:180:1);
	
	\tdplotsetrotatedcoords{0}{90}{90};
	\draw[dashed,
	tdplot_rotated_coords,
	gray
	] (1,0,0) arc (0:180:1);
	
	
	\draw[-stealth] (0,0,0) -- (1.80,0,0);
	\draw[-stealth] (0,0,0) -- (0,1.30,0);
	\draw[-stealth] (0,0,0) -- (0,0,1.30);
	\draw[dashed, gray] (0,0,0) -- (-1,0,0);
	\draw[dashed, gray] (0,0,0) -- (0,-1,0);
	
	\draw[dashed, -stealth, blue!50] (0,0,0) -- (P) node[left] {$v_1$};
	\draw[dashed, -stealth, blue!50] (0,0,0) -- (Q) node[left] {$v_2$};;
	\draw[dashed, -stealth, blue!50] (0,0,0) -- (R) node[left] {$v_4$};;
	\draw[dashed, -stealth, blue!50] (0,0,0) -- (S) node[right] {$-v_3$};;
	\draw[thick , blue!50] (Q) -- (P);
	\draw[thick , blue!50] (Q) -- (R);
	\draw[thick , blue!50] (Q) -- (S);
	\draw[thick , blue!50] (P) -- (S);
	\draw[thick, blue!50] (P) -- (R);
	\draw[thick, blue!50] (R) -- (S);

	\draw[dashed, -stealth, red!50] (0,0,0) -- (P');
	\draw[dashed, -stealth, red!50] (0,0,0) -- (Q');
	\draw[dashed, -stealth, red!50] (0,0,0) -- (R');
	\draw[dashed, -stealth, red!50] (0,0,0) -- (S');
	
	\draw[fill = lightgray!50] (P) circle (0.5pt);
	\draw[fill = lightgray!50] (Q) circle (0.5pt);
	\draw[fill = lightgray!50] (R) circle (0.5pt);
	\draw[fill = lightgray!50] (S) circle (0.5pt);
	\draw[fill = lightgray!50] (P') circle (0.5pt);
	\draw[fill = lightgray!50] (Q') circle (0.5pt);
	\draw[fill = lightgray!50] (R') circle (0.5pt);
	\draw[fill = lightgray!50] (S') circle (0.5pt);
	\end{tikzpicture}

	\caption[]{Left: Cube Type, Right: Triangular Pyramid Type.  Notice that we flip the sign of $v_3$ in the right subfigure for convention such that $v_1 + v_2 + v_3 + v_4$ attains the maximum of $\| \sum_{i=1}^4 \delta_i v_i \|^2$.} \label{fig: two types} 
\end{figure}
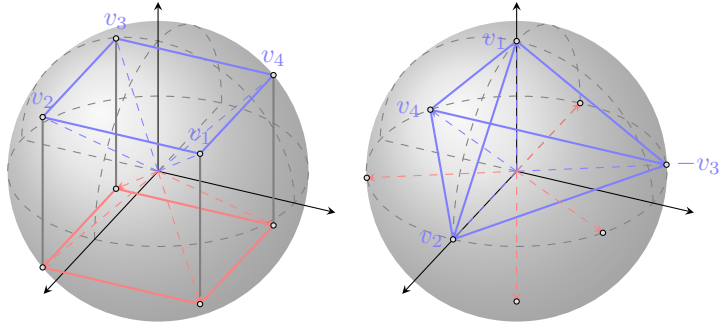  

\begin{prop}\label{prop: case3}
	Suppose $p=3, ~ n=4$.
	If $l$ attains its local minimal at $\ul{v}=(v_1,v_2,v_3,v_4)$ where $\spn\{v_1, \cdots, v_4\} = \bR^3$ and $v_i \neq \pm v_j$ for any two indices $i\neq j$, then 
	$\ul{v}$ is either of cube type or of triangular pyramid type.
	If $\ul{v}$ is of cube type, $l(\ul{v})=\frac{16}{3}$.
	If $\ul{v}$ is of triangular pyramid type, $l(\ul{v})=5$.
	In particular, $l$ attains its global minimal at triangular pyramid type configurations, and 
	$$\min_{\substack{\ul v \in (S^2)^4}} \max_{\substack{\delta \in \{\pm 1\}^4}} \| \sum_{i=1}^4 \delta_i v_i \|^2 = 5.$$
\end{prop}

\begin{proof}
Firstly, straightforward calculation verifies the value of $l$ under the cube type equals $\frac{16}{3}$, and the value of $l$ under the triangular pyramid type equals $5$. 
Now we show that  $\ul v = (v_1,v_2,\cdots, v_4)$ will not attain the global minimal of function $l$ if $\dim(\spn(v_1, v_2,v_3,v_4) < 3$ or $v_i = \pm v_j$ for some indices $i,j$. Suppose $\dim(\spn(v_1, v_2,v_3,v_4) < 3$, then the problem reduces to Case 2 as discussed in Section \ref{subsec: case 2}, and we know the minimal value of $l$ equals $\sin(\pi/8)^{-2} \approx 6.828$. Suppose $v_i = \pm v_j$ for some $i,j$, we may assume without loss of generality that $v_1 = v_2$, then we claim the minimal of $l$ under this extra assumption ($v_1 = v_2$) equals $6 = \lVert 2v1 \lVert^2  + \lVert v_2 \lVert ^2 + \lVert v_3 \lVert^2$, and the minimum is attained when $v_1, v_3, v_4$ are mutually orthogonal. In other words,

$$\min_{\substack{\ul v \in (S^2)^4\\ v_1 = v_2}} \max_{\substack{\delta \in \{\pm 1\}^4}} \| \sum_{i=1}^4 \delta_i v_i \|^2 = 6.$$
The proof is essentially the same as Proposition \ref{prop:case1}. 
  Since values are larger than $5$ -- the function value of $l$ under the triangular pyramid configuration,  we may assume without loss of generality that $\spn\{v_1, \cdots, v_4\} = \bR^3$ and $v_i \neq v_j$ for any $i\neq j\in \{1,2,3,4\}$.

Next, suppose $\ul v$ attains the local minimum of $l$, we study the cardinality of the set $M_{\ul v}$. We  make the following claim, which will be proved at the end of this section.
\begin{claim}\label{claim:cardinality}
 Suppose $\ul v = (v_1, v_2, v_3, v_4)$ satisfies the assumption in Proposition \ref{prop: case3}, then $|M_{\ul{v}}| \geq 3 $.
\end{claim}

Given $|M_{\ul{v}}| \geq 3 $, we now discuss two possible cases for $M_{\ul{v}}$ separately. The two cases  eventually correspond to the cube design and the triangular pyramid design, as we will see shortly. 
\begin{itemize}
	\item Case 1: Every pair of elements in $M_{\ul v}$ are differed by exactly two indices. In other words,
	for any $\alpha$ and $\beta$ in $M_{\ul{v}}$, there are exactly two indices $1 \leq i_1 \neq i_2 \leq 4$ such that $\alpha_{i_1}= -\beta_{i_1}, ~ \alpha_{i_2}= -\beta_{i_2}$.
	\item Case 2: The complement of Case 1. In other words,  there exist $\alpha$ and $\beta$ in $M_{\ul{v}}$ such that they  differ by one or three indices. 
\end{itemize}
If $M_{\ul{v}}$ satisfies Case 1, we can make suitable relabelling such that
$$\{(1,1,1,1), ~ (1,1,-1,-1), ~(1,-1,1,-1)\} \subset M_{\ul{v}}$$ and therefore 
\begin{equation}\label{eqn: three equalities}
\| v_1+v_2+v_3+v_4 \|^2 = \| v_1+v_2-v_3-v_4 \|^2 = \| v_1-v_2+v_3-v_4 \|^2.
\end{equation}
Expanding \ref{eqn: three equalities} yields
\begin{equation}\label{e2}
\langle v_1, v_2 \rangle + \langle v_3, v_4 \rangle = \langle v_1, v_3 \rangle + \langle v_2, v_4 \rangle = -(\langle v_1, v_4 \rangle + \langle v_2, v_3 \rangle).
\end{equation}
Now we  further claim  $|M_{\ul{v}}| = 3$ in Case 1, as otherwise by the same argument we have:
\begin{equation}\label{eqn: four equalities}
\| v_1+v_2+v_3+v_4 \|^2 = \| v_1+v_2-v_3-v_4 \|^2 = \| v_1-v_2+v_3-v_4 \|^2 = \| v_1 - v_2 - v_3 + v_4 \|^2.
\end{equation}
Expanding \ref{eqn: four equalities} and summing up the four terms  cancels out all the cross-terms and gives us $\| v_1+v_2+v_3+v_4 \|^2 = 4$, which implies $\mathcal E_{4,3}\leq \sqrt 4 = 2$. However, 
by the averaging trick in the proof of Proposition \ref{prop:case1}, we know $\mathcal E_{4,3}\geq 2$
where the inequality holds if and only if the four vectors $v_1, v_2, v_3, v_4$ are mutually orthogonal. In more details, we have an invariant \[\sum_{\substack{\delta \in \{\pm 1\}^4}} \| \sum_{i=1}^4 \delta_i v_i \|^2 = 4 \times 2^4. \] If $\mathcal E_{4,3} = 2$, by the pigeonhole principle, $\| \sum\limits_{i=1}^4 \delta_i v_i \|^2 = 4$ for all $\delta \in \{\pm 1\}^4$, and $v_i \cdot v_j = 0$ for all $i \neq j$. This contradicts with the setting $p = 3$. Therefore $|M_{\ul{v}}| = 3$, as claimed.

For now, we define the following notation. 
$$M_1:=v_1+v_2+v_3+v_4, ~ M_2:=v_1+v_2-v_3-v_4, ~ M_3:=v_1-v_2+v_3-v_4.$$
By Formula \ref{eqn: non-surjective map} in the proof of Lemma \ref{lem: local minimal}, there exists a vector $(x,y,z) \neq (0,0,0)$ such that 
\begin{equation}\label{eqn: non-surjective, 3d}
x\langle M_1, t_1+t_2+t_3+t_4\rangle+y\langle M_2, t_1+t_2-t_3-t_4\rangle+z\langle M_3, t_1-t_2+t_3-t_4\rangle=0
\end{equation}
 for any $t_i$ in the tangent space of $v_i$ on $S^2$.
Expanding \ref{eqn: non-surjective, 3d} and collecting terms with respect to $t_i$ yields
\begin{equation*}
xM_1+yM_2+zM_3 \parallel v_1, ~ xM_1+yM_2-zM_3 \parallel v_2, ~ xM_1-yM_2+zM_3 \parallel v_3, ~ xM_1-yM_2-zM_3 \parallel v_4,
\end{equation*}
where $v \parallel w$ means that $v$ is parallel to $w$ in the usual Euclidean space.

Setting $a=x+y-z, ~ b=x-y+z, ~ c=x-y-z$, the parallel relationship  is further equivalent to 
$$a v_2+b v_3+c v_4 \parallel v_1, ~ a v_1+c v_3+b v_4 \parallel v_2, ~ b v_1+c v_2+a v_4 \parallel v_3, ~ c v_1+b v_2+a v_3 \parallel v_4.$$
As $\{v_1,v_2,v_3,v_4\}$ spans the whole space $\bR^3$, there exists $(p_1,p_2,p_3,p_4) \neq (0,0,0,0)$ which is unique up to a scalar multiple such that $p_1 v_1+p_2 v_2+p_3 v_3+p_4 v_4=\vec{0}$.
Therefore by the parallel relations,
$$[p_2:p_3:p_4]=[a:b:c], ~ [p_1:p_3:p_4]=[a:c:b],$$
$$[p_1:p_2:p_4]=[b:c:a], ~ [p_1:p_2:p_3]=[c:b:a].$$
From the above relation we can derive that $[a:b:c] \in \{[1:1:1], [1:1:-1], [1:-1:1], [1:-1:-1]\}.$
By the fact that $M_i$ are nonzero vectors, the only possibility is that
$$[a:b:c]=[1:1:-1], ~ v_1-v_2-v_3+v_4=\vec{0}.$$
Combining $\| v_1-v_2-v_3+v_4 \|=0$, together with equation (\ref{e2}) and the fact that $\|v_4 \|=1$, one can solve the inner products of any pairs  of $\{v_1,v_2,v_3,v_4\}$. The Gram matrix $G := (\langle v_i, v_j \rangle)_{1\leq i,j\leq 4}$ can be calculated as:
\[
G = \begin{pmatrix}
1 & 1/3& 1/3& -1/3\\
1/3& 1 & -1/3 & 1/3\\
1/3& -1/3& 1 & 1/3\\
-1/3& 1/3& 1/3& 1\\ 
\end{pmatrix},
\] 
which corresponds to the cube type.

\vspace{5mm}
Otherwise, $M_{\ul{v}}$ satisfies case 2. Suitable relabelling allows us to assume that \hfill\break
$(1,1,1,1), ~ (1,1,-1,-1), ~ (1,-1,-1,-1) \in M_{\ul{v}}$. Similarly, we define
$$M_1:=v_1+v_2+v_3+v_4, ~ M_2:=v_1+v_2-v_3-v_4, ~ M_3:=v_1-v_2-v_3-v_4.$$ 

In contrary to Case 1,  we will show  $|M_{\ul{v}}| \geq 4$ in Case 2.  Suppose $|M_{\ul{v}}|=3$, similar to case 1, Lemma \ref{lem: local minimal} guarantees the existence  of a non-zero vector $(x,y,z)$ such that 
\begin{equation*}
x\langle M_1, t_1+t_2+t_3+t_4\rangle+y\langle M_2, t_1+t_2-t_3-t_4\rangle+z\langle M_3, t_1-t_2-t_3-t_4\rangle=0,
\end{equation*}
 for any $t_i$ in the tangent space of $v_i$ on $S^2$. Equivalently, we have

\begin{equation}\label{eqn: parallel relation, case2}
xM_1+yM_2+zM_3 \parallel v_1, ~ xM_1+yM_2-zM_3 \parallel v_2, ~ xM_1-yM_2-zM_3 \parallel v_3, v_4.
\end{equation}

By our assumption, two lines spanned by $v_3$ and $v_4$ are distinct, hence the third parallel relationship in \ref{eqn: parallel relation, case2} shows 
\begin{equation}\label{eqn:combination equals 0}
xM_1-yM_2-zM_3=\vec{0}.
\end{equation}
Plugging \ref{eqn:combination equals 0} back into the first two parallel relationship in \ref{eqn: parallel relation, case2} shows

 $$xM_1 \parallel v_1, ~ yM_2 \parallel v_2.$$
If $y=0$, one deduces $M_1 \parallel M_3 \parallel v_1$.
Since  $\|M_1\|=\|M_3\|$ by definition,
we have
 \begin{eqnarray*}
 v_2+v_3+v_4=\vec{0}, ~ \|M_1\|=\|M_3\|=1,
\end{eqnarray*}
which is impossible.
Similarly, the case $x=0$ can be ruled out. 
Therefore it suffices to discuss the case where both $x$ and $y$ are nonzero. Since we have
$$M_1 \parallel v_1, ~ M_2 \parallel v_2.$$
Write $M_1 = \lambda_1 v_1$ and $M_2 = \lambda_2 v_2$, we can use the relationship
\[
(1-\lambda_1) v_1 + v_2 + v_3 + v_4 = v_1 + (1-\lambda_2)v_2 -v_3 -v_4 = 0
\]
to solve $\lambda_1 = 2, \lambda_2 = 2$, therefore 
 $$v_3+v_4=v_1-v_2,$$
 and $M_3 =0$, which is also a contradiction. Therefore we know $|M_{\ul{v}}| \geq 4$.

Given $|M_{\ul{v}}| \geq 4$,  we now discuss on the fourth vector in $M_{\ul{v}}$ other than\\ $(1,1,1,1), ~ (1,1,-1,-1), ~ (1,-1,-1,-1) \in M_{\ul{v}}$.  Firstly, if $(1,-1,1,1) \in M_{\ul{v}}$, from 
$$\|M_1\|=\|M_2\|=\|M_3\|=\|v_1-v_2+v_3+v_4\|$$
we have
$$ v_1 \perp v_2 \perp v_3+v_4.$$
It can  be directly checked that $\ul{v}$ is of the triangular pyramid type.
Let $e_3 = v_1 \times v_2$. Suppose $v_3 =  a v_1 + b v_2 + x e_3$, then $v_4 = -a v_1 -b v_2 + x e_3$ for $(a, b, x)\in S^2$.
Therefore $v_3 - v_4 = 2(a v_1 + b v_2)$. We also know $\lVert M_i \rVert^2 = 1^2 + 1^2 + (2x)^2 = 2 + 4x^2$ for any $i$,
by the maximal property of $M_i$, we have, $$\|v_1 + v_2 \pm 2(a v_1 + b v_2)\|^2 = 4a^2 + 4b^2 \pm4(a+b) + 2 \leq 2 + 4x^2 = 2 + 4(1-a^2-b^2)$$
$$\|v_1 - v_2 \pm 2(a v_1 + b v_2)\|^2 = 4a^2 + 4b^2 \pm4(a-b) + 2 \leq 2 + 4x^2 = 2 + 4(1-a^2-b^2).$$
We have $$\max{2(a^2+b^2)\pm(a+b),2(a^2+b^2)\pm(a-b)} \leq 1.$$
Without loss of generality, $a$, $b \geq 0$.
$$2(a^2+b^2)+\sqrt{a^2+b^2} \leq 2(a^2+b^2) + (a+b) \leq 1,$$
hence $a^2 + b^2 \leq \frac{1}{4}$, therefore $l(\ul v)^2 = \lVert M_i \rVert^2 \geq  2 + 4x^2 = 2 + 4 (1 - a^2 - b^2) \geq 2 + 4 \times \frac 34 = 5$, and the inequality attains equality when $\ul v$ is of the triangular pyramid type. In other words, up to reflections and index permutations, the Gram matrix
 $G := (\langle v_i, v_j \rangle)_{1\leq i,j\leq 4}$ is given by:
 \[
 G = \begin{pmatrix}
 1 & 0& 0& 0\\
 0& 1 & 1/2 & -1/2\\
 0& 1/2& 1  & 1/2\\
 0 & -1/2& 1/2& 1
 \end{pmatrix}.
 \]

The remaining cases can be argued using a similar but slightly more complicated way. Suppose that $(1,1,-1,-1) \notin M_{\ul v}$, we know up to equivalence that either $(1,1,1,-1)$ or $(1,-1,1,-1)$ is in $M_{\ul{v}}$.  We will do the case where $(1,-1,1,-1) \in M_{\ul{v}}$ by contradiction, and the other case can be proved in the same way. 
 
Suppose $(1,-1,1,-1) \in M_{\ul{v}}$, we claim that $|M_{\ul{v}}| \geq 5$. Otherwise,  we can again write $M_4:=v_1-v_2+v_3-v_4$.  It can be shown from Lemma \ref{lem: local minimal} that there exists $(x,y,z,w) \neq (0,0,0,0)$ such that 
$$xM_1+yM_2+zM_3+wM_4 \parallel v_1, ~ xM_1+yM_2-zM_3-wM_4 \parallel v_2,$$
$$xM_1-yM_2-zM_3+wM_4 \parallel v_3, ~ xM_1-yM_2-zM_3-wM_4 \parallel v_4.$$
Set $$a=x+y-z-w, ~ b=x+y+z-w, ~ c=x-y+z+w,$$
$$d=x-y+z-w, ~ e=x-y-z+w, ~ f=x-y-z-w.$$
The parallel relations are then translated to
$$av_2+ev_3+fv_4 \parallel v_1, ~ av_1+dv_3+cv_4 \parallel v_2,$$
$$ev_1+dv_2+bv_4 \parallel v_3, ~ fv_1+cv_2+bv_3 \parallel v_4.$$
Using the nondegeneracy of $\{v_1,v_2,v_3,v_4\}$ (there exists a unique vector $(p_1,p_2,p_3,p_4)$  up to a scalar multiple such that $p_1v_1+p_2v_2+p_3v_3+p_4v_4=\vec{0}$), we deduce that
$$ab=ce=df \Rightarrow (x+y-w)^2=(x-y+w)^2=(x-y-w)^2$$
$$\Rightarrow (x,y,w)=(1,1,1) ~\mathrm{or}~ (1,0,0) ~\mathrm{or}~ (0,1,0) ~\mathrm{or}~ (0,0,1)$$ up to a scalar multiple.
If $(x,y,w)=(1,1,1)$, by the parallel relations, we deduce that $$z=0, ~ v_1+v_4=v_2+v_3,$$
which means
$$ M_3=2v_1, ~ |M_3|^2=4, ~ l(\ul{v})=4,$$
we have discussed in Proposition \ref{prop:case1} that  this is equivalent to the case that $v_i$'s are orthogonal to each other, yielding a contradiction.
If $(x,y,w)=(1,0,0)$ (and similarly for the rest two cases),
$$M_1-zM_3=\vec{0}, ~ M_1 \parallel M_3 \parallel v_1.$$
Since $\|M_1\|=\|M_3\|$, $v_2+v_3+v_4=\vec{0}$, we conclude $l(\ul{v})=\|M_1\|=\|M_3\|=1$, which is also a contradiction.

Given $|M_{\ul{v}}| \geq 5$, $\{(1,1,1,1),(1,1,-1,-1),(1,-1,-1,-1),(1,-1,1,-1)\} \subset M_{\ul{v}}$, and $(1,-1,-1,1) \notin M_{\ul{v}}$ (otherwise $v_i$'s are all orthogonal to each other).
In all the other cases, $M_{\ul{v}}$ contains four elements $\gamma, \delta, \sigma, \tau$ such that $\gamma_{i_1}=\gamma_{i_2}, ~ \delta_{i_1}=\delta_{i_2}, ~ \sigma_{i_1}=\sigma_{i_2}, ~ \tau_{i_1}=\tau_{i_2}$ for two different indices $1 \leq i_1 \neq i_2 \leq 4$, reducing to the situation that $v_1 \perp v_2 \perp v_3+v_4$, and as discussed before, corresponding to the triangular pyramid type.
\end{proof}

We conclude the proof of Proposition \ref{prop: case3} by showing Claim \ref{claim:cardinality}.

\begin{proof}[Proof of Claim \ref{claim:cardinality}]
First, we claim $|M_{\ul{v}}| \geq 2$. Suppose the contrary, since Lemma \ref{lem: local minimal} shows the linear map from $\prod\limits_{i=1}^n T_i$ to $\bR^{|M_{\ul{v}}|} = \bR$ is not surjective, we immediately have  $D$ is the zero map, a clear contradiction.

Suppose $|M_{\ul v}| = 2$ and $M_{\ul v} = \{\alpha, \beta\}$. The case where $\{V_\alpha, V_\beta\}$ is linear independent has already been excluded by Lemma \ref{lem: local minimal} as well, since it is shown that the non-zero linear combinations of $\{V_\alpha, V_\beta\}$ will generates $\spn\{v_1, v_2, v_3, v_4\}$, which is of dimension $3$, a contradiction. 

It only remains to discuss the case where $\{V_\alpha, V_\beta\}$ is linearly dependent. Since $\| V_\alpha \| = \|V_\beta\|$,  we have $V_\alpha = \pm V_\beta$. We can assume $V_\alpha = V_\beta$ as otherwise we may simply choose $M_{\ul v} = \{\alpha, -\beta\}$. Again, after suitable relabelling we can  assume $V_\alpha = v_1 + v_2 + v_3 + v_4$, and $V_\beta$ is either $v_1 + v_2 + v_3 - v_4$ or $v_1 + v_2 - v_3 - v_4$ or $v_1 - v_2 - v_3 - v_4$.
For the first case $v_4 = 0$. For the second case $v_3$, $v_4$ spans the same line (contradicts with the setting of Proposition \ref{prop: case3}). For the third case $V_\beta = v_1$ with unit length which contradicts with Proposition \ref{prop:case1}. All the cases are excluded and we conclude $|M_{\ul v}|\geq 3$.
\end{proof}

Combining Proposition \ref{prop:case1}, \ref{prop:case2}, \ref{prop: case3}, the following theorem is immediate. 

\begin{thm}\label{thm: exact design}
	The minimal energy defined in (\ref{eqn: minimal energy}) and the corresponding spherical minimax design can be explicitly derived in the following three cases:
	
	\begin{itemize}
	\item
\textit{Case 1: $ p \geq n$}, the minimax design is the set of $n$ mutually orthogonal vectors  with $\mathcal E_{n,p} = \sqrt{n}$.
	\item Case 2: $p = 2, n$, the minimax design is  the evenly spaced points on the upper semi-circle with $\mathcal{E}_{n,2} = \sin^{-1}(\frac{\pi}{2n})$,
	\item  Case 3: $p = 3, n = 4$, the minimax design is of the  so-called triangular pyramid type (see Definition \ref{def: triangular pryamid}) with $\mathcal{E}_{3,4} = \sqrt {5}$.
\end{itemize}
\end{thm}
\begin{proof}
	 Combining  proposition \ref{prop:case1}, \ref{prop:case2}, and \ref{prop: case3} and Theorem \ref{thm: exact design} automatically follows.
\end{proof}

Alas, we find our method very difficult to generalize to  other cases such as $p = 3$ and $n = 5$. It seems that finding the exact minimax spherical designs for general $n, p$ is a particularly challenging task. Instead of giving the exact results, we will study the asymptotic behaviors of $\ener$ in the next section.

\section{Asymptotic Results}\label{sec: asymptotic energy}

We are interested in the asymptotic behavior of the  quantity:

\begin{align}
\ener &:= \min_{\ul{u} \in (\Sp)^n} \Enp(\ul{u}) =  \min_{u_1, \cdots, u_n \in \Sp}   \max_{v\in \Sp} \sum_{i = 1}^n \lvert u_i \cdot v\rvert
\end{align}
under different regimes. We assume $n > p$ henceforth as otherwise the problem is solved in Section \ref{sec: exact designs}, Case 1. Before stating and proving our main results, we introduce two auxiliary lemmas. The first lemma shows some basic properties of a random variable uniformly distributed on $\Sp$.

\begin{lem}[Distribution of the first coordinate on the $p$-sphere]\label{lem: distribution, uniform sphere}
	Let $v$ be a random variable which is uniformly distributed on the $\Sp$, then the first coordinate $v_1$ has the following probability density function on $[-1,1]:$  
	\begin{align}\label{eqn: first coordinate distribution}
	f_{v_1}(s) = \frac{(1-s^2)^{\frac p 2 - \frac 32}}{B(\frac {p-1}2, \frac 12)}, 
	\end{align}
	where $B$ is the Beta function.
	Moreover, for any fixed $u\in \Sp$, 
	\begin{align}\label{eqn: first coordinate expectation}
	\bE_{v \sim \Unif(\Sp)}\lvert v\cdot u \rvert = 	\bE_{v \sim \Unif(\Sp)}\lvert v_1 \rvert = \frac{\Gamma(p/2)}{\sqrt{\pi}   \Gamma((p+1)/2)}  
	\end{align}
\end{lem}

\begin{proof}
	For the first part, let $Z_1, Z_2, \cdots, Z_p$ be independent and identically distributed (i.i.d.)  standard normal random variables. It is well known that the following random vector:
	\[
	\bigg(\frac{Z_1}{\sqrt{\sum_{i = 1}^p Z_i^2}}, \frac{Z_2}{\sqrt{\sum_{i = 1}^p Z_i^2}} \cdots, \frac{Z_p}{\sqrt{\sum_{i = 1}^p Z_i^2}}\bigg)^\intercal
	\]
	is uniformly distributed on the sphere. Therefore, 
	\begin{align}\label{eqn:F-distribution}
	\bP\bigg(\bigg\lvert \frac {Z_1}{\sqrt{\sum_{i = 1}^p Z_i^2}}\bigg\rvert\leq s\bigg) = \bP\bigg( \frac {\sum_{i = 2}^p Z_i^2} {Z_1^2}\geq \frac 1{s^2} - 1\bigg) =  \bP\bigg( \frac {\sum_{i = 2}^p Z_i^2} {(p-1)Z_1^2}\geq \frac{( 1/{s^2} - 1)}{p-1}\bigg),
	\end{align}
	where the RHS of (\ref{eqn:F-distribution}) can be expressed by the CDF of the $F_{p-1,1}$ distribution which has known density function. Taking the derivative of (\ref{eqn:F-distribution}) with respect to $s$ and (\ref{eqn: first coordinate distribution})  follows.
	
	To prove (\ref{eqn: first coordinate expectation}), we observe that the uniform distribution on $\Sp$ is rotational invariant, therefore the quantity $\bE_{v \sim \Unif(\Sp)}\lvert v\cdot u \rvert$ does not depend on $u$. We may simply choose $u = e_1 = (1, 0, \cdots,0)^\intercal$ which proves the first equality of (\ref{eqn: first coordinate expectation}). The second equality of (\ref{eqn: first coordinate expectation}) are straightforward.
\end{proof}

Let $\sigma_p$ be the standard Euclidean Lebesgue measure on the unit sphere $\Sp$. Let $\{\cR_1, \cR_2, \cdots, \cR_n\}$ be a disjoint collection  such that $\cR_i \subset \Sp$ for each $i$. The collection is called an \textit{area-regular partition} if $\cup_i \cR_i = \Sp$ and $\sigma_p(\cR_i) = \frac{\sigma_p(\Sp)}{n} $ for every $i$.

The second auxiliary lemma is about the area regular partitions of $\Sp$, see \cite{bourgain1988distribution} \cite{kuijlaars1998asymptotics} for proofs.
\begin{lem}[Area-regular partition]\label{lem: area-regular partition}
	For each $n, p \in \bN$, there exists an		area-regular partition $\{\cR_1, \cR_2, \cdots, \cR_n\}$ of the unit sphere $\Sp$ such that:
	$$\max_{i} \text{diam~} \cR_i \leq C_p n^{-1/p},$$
	where $C_p$ is a constant  depending only on $p$, $\text{diam~} \cR_i := \max_{x,y\in \cR_i} \|x-y\|$.
\end{lem}

With all the lemmas in hand, now we are ready to prove the asymptotic results of $\ener$. We first consider the case that $p$ is a fixed positive integer and $n$ goes to infinity.

\begin{thm}[Asymptotics for  $p$ fixed, $n\rightarrow \infty$]\label{thm: asymptotics, fix p}
	With all the notations as above, we have the following: 
	\begin{align} \label{eqn: e(n,p)}
	\frac{\Gamma(p/2)}{\sqrt{\pi}   \Gamma((p+1)/2)} \leq & \frac{\ener}{n}\leq \frac{\Gamma(p/2)}{\sqrt{\pi}   \Gamma((p+1)/2)}  + C_p n^{-\frac 1p}.
	\end{align}
\end{thm}

The above result shows $\ener$ grows linearly with $n$ at the rate of $\frac{\Gamma(p/2)}{\sqrt{\pi}   \Gamma((p+1)/2)}$. The proof relies on a probabilistic argument. More precisely, we aim to show the following:

\[
\ener \approx \bE_{v, \ul{u}}(\sum_{i = 1}^n \lvert u_i\cdot v \rvert),
\]
where $v, u_1, \cdots, u_n$ are independent and identically distributed (i.i.d.) uniform random variables on $\Sp$.

\begin{proof}
	We start with proving the lower bound of (\ref{eqn: e(n,p)}), observe that for any $\ul u \in (\Sp)^n$, 
	\[
	\Enp(\ul u) = \max_{v\in \Sp} \sum_{i = 1}^n \lvert u_i \cdot v\rvert \geq \bE_{v \sim \Unif(\Sp)} (\sum_{i = 1}^n \lvert u_i \cdot v \rvert) = \sum_{i = 1}^n\bE_{v \sim \Unif(\Sp)}(\lvert u_i \cdot v\rvert) .
	\]
	In view of Lemma \ref{lem: distribution, uniform sphere}, we have:
	\[
	\bE_{v \sim \Unif(\Sp)}  \lvert u_i \cdot v \rvert = \bE_{v \sim \Unif(\Sp)}  \lvert e_1 \cdot v \rvert = \frac{\Gamma(p/2)}{\sqrt{\pi}   \Gamma((p+1)/2)},
	\]
	where $e_1 = (1, 0, \cdots, 0)^\intercal$. It is then clear that 
	\[
	\Enp(\ul u) \geq n \frac{\Gamma(p/2)}{\sqrt{\pi}   \Gamma((p+1)/2)},
	\]
	for any $\ul u$.
	Taking infimum over $\ul u \in (\Sp)^n$ yields
	\[
	\ener \geq n \frac{\Gamma(p/2)}{\sqrt{\pi}   \Gamma((p+1)/2)},
	\]
	which proves the LHS of (\ref{eqn: e(n,p)}).

	To prove the RHS of \ref{eqn: e(n,p)}, let $\{\cR_1, \cR_2, \cdots, \cR_n \}$ be the area-regular partition given by Lemma \ref{lem: area-regular partition}.  For each $i$, we pick an arbitrary $u_i \in \cR_i$. Then it is clear that
	\begin{align}
	\ener \leq \max_{v\in \Sp} \sum_{i = 1}^n \lvert u_i \cdot v\rvert. 
	\end{align}
	On the other hand, for each fixed $v_0\in \Sp$,
	\begin{align}\label{eqn: v(n,p), upper bound}
	\bE_{w\sim \Unif(\Sp)}	\lvert w\cdot v_0\rvert & = \frac{\int_{\Sp} |w\cdot v_0|  \sigma_p(dw) } {\sigma_p(\Sp)} = \frac{\sum_{i=1}^n\int_{\cR_i} |w\cdot v_0|  \sigma_p(dw) } {\sigma_p(\Sp)}\\
	& = \frac{\sum_{i=1}^n \bE_{w\sim \Unif(\cR_i)}\lvert w\cdot v_0 \rvert }{n}.
	\end{align}
	For each $w\in \cR_i$,  we have:
	\begin{align*}
	\bigl\lvert \lvert w\cdot v_0 \rvert  - \lvert u_i \cdot v_0 \rvert \bigr\rvert \leq \lvert (w-u_i)\cdot v_0 \rvert \leq \text{diam~} \cR_i \leq C_p n^{-1/p}
	\end{align*}
	in view of the triangle inequality and Cauchy-Schwarz inequality. Therefore, 
	\begin{align*}
	\lvert u_i \cdot v_0\rvert \leq \bE_{w\sim \Unif(\cR_i)}\lvert w\cdot v_0 \rvert + C_p n^{-1/p},
	\end{align*}
	and $\sum_{i=1}^n \lvert u_i \cdot v_0\rvert $ can be upper bounded by
	\begin{align*}
	\sum_{i=1}^n \lvert u_i \cdot v_0\rvert &\leq \big(\sum_{i=1}^n \bE_{w\sim \Unif(\cR_i)}\lvert w\cdot v_0 \rvert\big) + C_p n^{(p-1)/p} \\
	& =  n  \bE_{w\sim \Unif(\Sp)}	\lvert w\cdot v_0\rvert  + C_p n^{(p-1)/p}\\
	& = n\cdot \bigg( \frac{\Gamma(p/2)}{\sqrt{\pi}   \Gamma((p+1)/2)}  + C_pn^{-\frac 1p}\bigg).
	\end{align*}
	The above inequality holds for every $v_0\in \Sp$, thus taking supremum over $v_0$ yields
	\begin{align*}
	\ener\leq \Enp(\ul u) = \max_{v\in \Sp} \sum_{i = 1}^n \lvert u_i \cdot v\rvert \leq n\cdot \bigg( \frac{\Gamma(p/2)}{\sqrt{\pi}   \Gamma((p+1)/2)}  + C_p(n^{-\frac 1p})\bigg),
	\end{align*}
	which completes the proof of the RHS of (\ref{eqn: e(n,p)}). 
\end{proof}

The above proof also gives us the construction of an asymptotically minimax design. The next corollary is immediate.

\begin{cor}\label{cor: asymptotic minimax design}
	Let $p$ be fixed, and $\{\cR_1, \cR_2, \cdots, \cR_n \}$ be an area-regular partition of $\Sp$ given by Lemma \ref{lem: area-regular partition}.  For each $i$, we pick an  $u_i \in \cR_i$ uniformly. Then $\ul {u^\star} := (u_1, u_2, \cdots, u_n)$ is an asymptotically minimax design. In other words, $\Enp(\ul {u^\star}) \rightarrow \ener$ as $n\rightarrow \infty$.
\end{cor}

If we allow both $n,p$ to be arbitrarily large, the next result shows $\ener$ is always at the magnitude of $\Theta(\frac n{\sqrt p})$.

\begin{thm}\label{thm: asymptotics, arbitrary p}
	Let $n,p$ be two arbitrary  positive integers with $n > p$, 
	\begin{align}\label{eqn: asymptotics, n, p}
	\sqrt \frac{2}{\pi} \cdot\frac{n} {\sqrt {p+1}}< \ener < \frac{\sqrt 5}{2}\cdot \frac{n}{\sqrt p}.
	\end{align}
	
\end{thm}
\begin{proof}
	We start with the lower bound in (\ref{eqn: asymptotics, n, p}). Theorem \ref{thm: asymptotics, fix p} shows $
	\ener \geq n \frac{\Gamma(p/2)}{\sqrt{\pi}   \Gamma((p+1)/2)} $ for any $n,p$. Using the  Gautschi's inequality 
	\begin{align*}
	\frac{\Gamma(x+1)}{\Gamma(x+s)} < (x+1)^{1-s} \qquad \text{if} ~~ x > 0, s\in (0,1)
	\end{align*}
	with $x = \frac{p-1}{2}$ and $s = \frac 12$, we have
	\begin{align*}
	n \frac{\Gamma(p/2)}{\sqrt{\pi}   \Gamma((p+1)/2)} > \frac{n}{\sqrt \pi} \sqrt {\frac{2}{p+1}},
	\end{align*}
	as desired. 
	
	For the upper bound, we write $ n = k p + r$ with $k\in \bN^+$ and $0\leq r\leq p-1$. For every $i \in \{1, \cdots, n\}$, we choose $u_i = e_{m_i} \in \Sp$ with $m_i = i \mod p$, where $e_k$ denotes the unit vector with all the entries zero except for a one on the $k$-th coordinate. In view of Lemma \ref{lem: discrete optimization}, the energy $\Enp (u_1, \cdots, u_n)$ can be calculated explicitly as:
	\begin{align*}
	\Enp (u_1, \cdots, u_n) &= \lVert \sum_{i=1}^n u_i\rVert = \sqrt {r (k+1)^2  + (p-r)  k^2 } = \sqrt{p k^2 + 2kr + r},
	\end{align*}
	which can be upper bounded by
	\begin{align*}
	\Enp (u_1, \cdots, u_n) &= \sqrt \frac{(pk+r)^2 + (p-r) r}{p} \leq \sqrt{ \frac{n^2 + \frac {p^2} 4}{p}} < \frac{\sqrt 5}{2} \cdot \frac n {\sqrt p},
	\end{align*}
	which concludes the proof.
\end{proof}

We conclude this section with  the following conjecture.
\begin{conj} \label{conj: average energy}
	Let $R_p(n) := \frac{\ener}{n}$ be the `average energy' of the minimax design on $\Sp$. For each fixed $p$, we conjecture: $R_p(n)$ is a non-increasing sequence with  $n$.
\end{conj}
There are several evidences supporting Conjecture \ref{conj: average energy}. Firstly, the first $p$ terms of $R_p(n)$ equals exactly $1$, while its limit equals $\frac{\Gamma(p/2)}{\sqrt{\pi}   \Gamma((p+1)/2)} < 1$. Secondly, it is not hard to show $R_p(2n) \leq R_p(n)$ for every $n$ as  $\mathcal E_{2n,p}$ is upper bounded by $2\ener$ (we can repeatly choose each vector in the minimax design of $\ener$ twice).  Lastly, all the existing non-asymptotic results in Section \ref{sec: exact designs} support our conjecture. When $p = 2$,  the results in Section \ref{subsec: case 2} confirms our conjecture. When $p= 3$, we  have $R_3(3) = \frac {\sqrt 3}{3} > R_3(4) = \frac{\sqrt 5}{4} $. 

\section{Applications}\label{sec: application}
\subsection{L1-Principal Component analysis}
Principal component analysis (PCA) is a widely-used technique in statistical analysis for dimension reduction. However, the standard L2-PCA approaches are known to suffer from  outliers. Let $X$ be a data matrix with $n$ observations and $p$ features, the first principal component (PC1) of the classical L2-PCA looks for a $p$-dimensional vector $w_{(1)}\in \Sp$ which maximizes the L2 norm:
\begin{equation}\label{eqn: L2PCA}
w_{(1)} := \argmax_{w\in \Sp} \lVert Xw \rVert_2 = \argmax_{w\in \Sp} w^\intercal X^\intercal X w.
\end{equation}

To increase the robustness of the PCA algorithm, one proposal is to maximize the L1 norm instead of the L2 norm, the first principal component of the L1-PCA can be similarly defined as:
\begin{equation}\label{eqn: L1PCA}
	v_{(1)} := \argmax_{v\in \Sp} \lVert Xv \rVert_1 = \argmax_{v\in \Sp} \sum_{i =1}^n \lvert v\cdot x_i \rvert,
\end{equation}
where $x_1, \cdots, x_n \in \bR^p$ are the rows of the data matrix $X$. It is clear that (\ref{eqn: L1PCA}) is precisely the new energy function we have defined in \ref{eqn: max energy}. L1-PCA is often  prefered  than L2-PCA when the dataset has outliers or corrupted observations. Applications include image reconstruction \cite{kwak2008principal}, robust subspace factorization \cite{ding2006r}, regression analysis \cite{mccoy2011two} and so on. Although immense progresses have been made in the study of L1-PCA methods, most of the existing results focus on proposing efficient and accurate algorithms for solving (\ref{eqn: L1PCA}), see \cite{mccoy2011two} \cite{kwak2008principal} \cite{markopoulos2017efficient} for examples. Our results are directly applicable to study the behavior of L1-PCA methods in the worst-case scenario. For example, suppose we have normalized all the observations such that $x_i \in \Sp$ for every $i$, then Theorem \ref{thm: asymptotics, fix p} and \ref{thm: asymptotics, arbitrary p} imply the following result directly.
\begin{prop}\label{prop: L1 PCA}
Let $X = [x_1^\intercal, x_2^\intercal, \cdots, x_n^\intercal]^\intercal \in \bR^{n\times p}$ be a normalized data matrix, then when $p$ is fixed and $n \rightarrow \infty$, we have
\begin{equation}\label{eqn: L1 PCA p fixed}
\frac{\Gamma(p/2)}{\sqrt{\pi}   \Gamma((p+1)/2)} \leq  \frac{\min\limits_{\{x_1, \cdots, x_n\}\in (\Sp)^n}\max\limits_{v\in \Sp}\lVert Xv \rVert_1 }{n}  \leq \frac{\Gamma(p/2)}{\sqrt{\pi}   \Gamma((p+1)/2)}  + C_p n^{-\frac 1p}.
\end{equation}

For arbitrary positive integers $n, p >0$, we have
\begin{equation}\label{eqn: L1 PCA p arbitrary}
\sqrt \frac{2}{\pi(p+1)}  \leq  \frac{\min\limits_{\{x_1, \cdots, x_n\}\in (\Sp)^n}\max\limits_{v\in \Sp}\lVert Xv \rVert_1 }{n}  \leq \sqrt \frac{5}{4p}.
\end{equation}

\end{prop}
The quantity $\frac{\max\limits_{v\in \Sp}\lVert Xv \rVert_1 }{n} \in [0,1]$ has natural statistical interpretations. It can be viewed as a measure for the proportion of the normalized data matrix $X$ explained by the first principal component, similar to the concept `Proportion of Variance Explained' (PVE) in L2-PCA.  In one extreme case (best case) where all the vectors lie on the same line, it is clear that the first principal component equals $x_1$ up to a sign flip. In this case we also have the ratio  $\frac{\max\limits_{v\in \Sp}\lVert Xv \rVert_1 }{n}$ equals $1$.  Proposition (\ref{prop: L1 PCA}) shows, under the worst-case scenario, the first principal component of L1-PCA can still explain $\Theta(\frac 1
{\sqrt p})$ of the original data. A natural follow-up problem is to consider the proportion of the original data explained by the next few principal components or ask for the number of principal components that contain a prefixed proportion of the data. We hope to answer these questions in our future works.

\subsection{Quasi-Monte Carlo for surface integrals on the unit sphere}

Numerical integration  is an important problem in many scientific areas. Given a bounded Riemannian manifold $M\subset \bR^p$ and an integration $I(f) := \int_{M} f(x) \sigma(dM)$ of interest, the standard Monte Carlo method samples independent and uniformly distributed points $x_1, \cdots, x_n$ on $M$, and estimate the integration by $
\hat{I} (f) := \frac{\sum\limits_{i = 1}^n f(x_i)}{n}$. By the Law of Large Numbers (LLN) and the Central Limit Theorem (CLT), the expected error of the Monte Carlo approximation is in the order of $O(n^{-1/2})$. 

When $M$ is taken to be the unit sphere $\Sp$, the spherical Quasi-Monte Carlo (QMC) seeks for $n$ points $\{u_1, \cdots, u_n\}$ on the unit sphere such that the error between the empirical average of $f(u_i)$ converges to $I(f)$ at a faster rate than the  baseline $O(n^{-1/2})$. It turns out that QMC designs are closely connected with the minimax spherical designs. Let $(u_1, u_2, \cdots, u_n)$ be the asymptotically minimax spherical design selected according to Corollary \ref{cor: asymptotic minimax design}. Then the following result from \cite{brauchart2014qmc} shows $(u_1, \cdots, u_n)$ are better than the Monte Carlo method under certain smoothness assumptions.
\begin{thm}[Theorem 24 in \cite{brauchart2014qmc}, reformulated] \label{thm: qmc design}
For  fixed $n,p$, let $(u_1, \cdots, u_n)\in (\Sp)^n$ be a set of points chosen as above. Let $\mathbb H^s(\Sp)$ be the Sobolev space with smoothness parameter $s$ of functions in $L^2(\Sp)$ (see \cite{evans:98} Chapter 5 for a detailed definition). Then the following holds for  $s \in (\frac p 2, \frac p2 + 1)$:
\begin{equation}\label{eqn:qmc convergence}
\frac{\beta'}{n^{s/p}} \leq \sqrt {\bE \bigg(\sup_{f\in\mathbb H^s(\Sp)} \big(\frac{\sum_{i=1}^n f(u_i)}{n} - I(f)\big)^2\bigg)} \leq \frac{\beta}{n^{s/p}}
\end{equation}
where $\beta$ and $\beta'$ are two positive constants depending on the $\mathbb H^s(\Sp)$ norm but not on $n$.
\end{thm}

In addition to the theoretical results that consider the worst-case scenario, we also provide numerical evidence showing the QMC designs can be significantly more accurate than the Monte Carlo methods. 

\begin{eg}[QMC design for on the unit sphere $S^2$]
The  asymptotically minimax spherical design can be efficiently implemented on $S^2$. For simplicity, we assume $n = (k+1)^2$ for some positive integer $k$. We can evenly partition both the $z$-axis and the longitudes  into $k$ pieces. Then the sphere are naturally partitioned into $(k+1)^2$ pieces by the $k^2$ intersections. It can be directly verified that each piece has the same area, and each piece has diameter less than $4\pi n^{-\frac 12}$. See also   Figure \ref{fig: equal-area partition} for illustrations. Therefore, by randomly choosing points on each piece, we get an asymptotically minimax spherical design of $S^2$. 

\begin{center}
	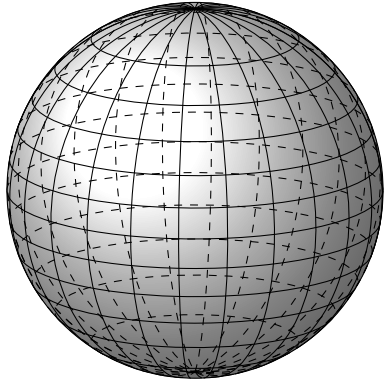
\begin{figure}[htbp]
	\begin{tikzpicture} 
	\def\R{2.5} 
	\def\angEl{15} 
	\filldraw[ball color= white] (0,0) circle (\R);
	\foreach \t in {-89, -56.44, -41.81, -30, -19.47, -9.59, 0, 
	9.59,19.47, 30, 41.81, 56.44,89} { \DrawLatitudeCircle[\R]{\t} }
	\foreach \t in {-5,-20,...,-185} { \DrawLongitudeCircle[\R]{\t} }
	\end{tikzpicture}
	\caption{An area-regular partition of $S^2$.}\label{fig: equal-area partition}
	\end{figure}
\end{center}

Here consider three functions, $f_1(\vec x) = x_1^2$,  $f_2(\vec x) = 1/ \lVert \vec x - (1,1,1) \rVert$, and $f_3(\vec x) = \exp(x_1 - x_2)$. The spherical surface integrals of each function can be evaluated analytically as below:
\begin{align*}
&\int_{S^2} f_1(x) \sigma(dx) = \frac{4\pi^2}{3},\\
&\int_{S^2} f_2(x) \sigma(dx) = \frac{4\pi^2}{\sqrt 3},\\
&\int_{S^2} f_3(x) \sigma(dx) = 2^{\frac 32} \pi\sinh(\sqrt 2).\\
\end{align*}

 Therefore, we  estimate each integral using both the Monte Carlo and the QMC methods and compare their performances. We choose  $n\in [10^4, 10^6]$, and implement both the Monte Carlo method and the Quasi-Monte Carlo method, each is repeated $50$ times for every fixed $n$. The Root Mean Square Error (RMSE)  of both methods are plotted below.  

\begin{center}
\begin{figure}[htbp]
\includegraphics[scale= 0.33]{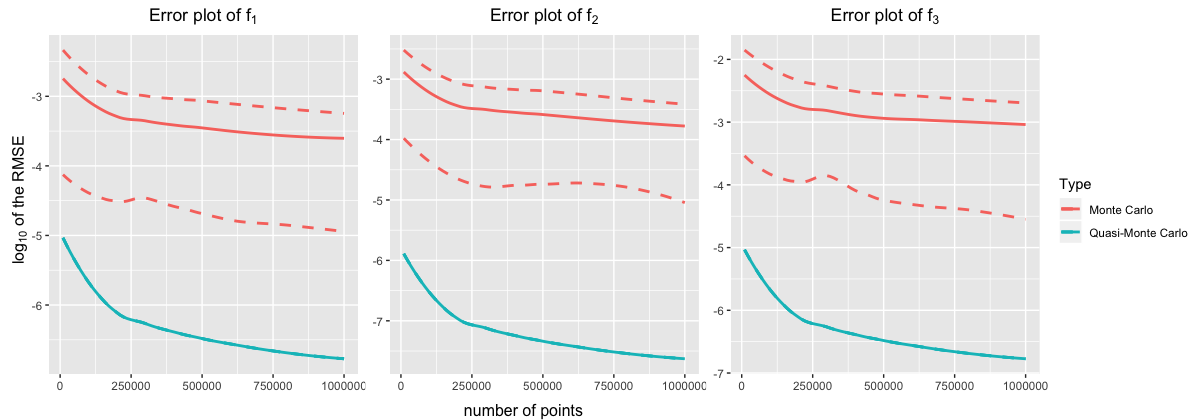}
\caption{Error plots of the Monte Carlo and the Quasi-Monte Carlo method for estimating the integral of $f_1, f_2, f_3$ on $S^2$. The horizontal axis stands for the number of points used for estimation. The vertical axis stands for the logarithm of the root mean square error under base $10$. Red and green solid lines correspond to the Monte Carlo and Quasi-Monte Carlo method, respectively. Red dotted lines are the $90\%$ confidence intervals of the Monte Carlo estimations based on $50$  independently repeated experiments.}\label{fig: mc vs qmc}
\end{figure}

\end{center}
Figure \ref{fig: mc vs qmc} suggests two important advantages of the QMC method, in contrast to the Monte Carlo method. Firstly, for all three test functions, QMC method offers several  orders of magnitude better accuracy than the Monte Carlo method. Secondly, QMC method converges to ground truth at an order of magnitude faster than the Monte Carlo method. For all  three test functions, when $n$ is increasing from $n_1 = 10^4$ to $n_2 = 100  n_1 = 10^6$, the error of the QMC method decreases to $\sim 1\%$ of the original, while the error using the Monte Carlo method only decreases to $\sim 10\%$ of the original.
\end{eg}

Both theoretical and empirical studies have shown promising results of the QMC method, but many challenges remain. Computationally, it is  unknown to us how to design efficient and implementable QMC designs when $p \gg 3$. Mathematically, Theorem \ref{thm: qmc design} concerns the convergence rate of a special asymptotically spherical minimax design. We do not know whether the exact spherical minimax design can achieve better convergence bounds than (\ref{eqn:qmc convergence}) or not.

\newpage
\bibliographystyle{amsplain}
\bibliography{refs}{}

\providecommand{\bysame}{\leavevmode\hbox to3em{\hrulefill}\thinspace}
\providecommand{\MR}{\relax\ifhmode\unskip\space\fi MR }
\providecommand{\MRhref}[2]{%
  \href{http://www.ams.org/mathscinet-getitem?mr=#1}{#2}
}
\providecommand{\href}[2]{#2}
\begin{thebibliography}{10}

\bibitem{alon2016probabilistic}
Noga Alon and Joel~H Spencer, \emph{The probabilistic method}, John Wiley \&
  Sons, 2016.

\bibitem{borodachov2008asymptotics}
S~Borodachov, D~Hardin, and E~Saff, \emph{Asymptotics for discrete weighted
  minimal riesz energy problems on rectifiable sets}, Transactions of the
  American Mathematical Society \textbf{360} (2008), no.~3, 1559--1580.

\bibitem{bourgain1988distribution}
J~Bourgain and J~Lindenstrauss, \emph{Distribution of points on spheres and
  approximation by zonotopes}, Israel Journal of Mathematics \textbf{64}
  (1988), no.~1, 25--31.

\bibitem{brauchart2014qmc}
Johann Brauchart, E~Saff, I~Sloan, and R~Womersley, \emph{{QMC designs: optimal
  order quasi Monte Carlo integration schemes on the sphere}}, Mathematics of
  computation \textbf{83} (2014), no.~290, 2821--2851.

\bibitem{chazelle2001discrepancy}
Bernard Chazelle, \emph{The discrepancy method: randomness and complexity},
  Cambridge University Press, 2001.

\bibitem{cohn2007universally}
Henry Cohn and Abhinav Kumar, \emph{Universally optimal distribution of points
  on spheres}, Journal of the American Mathematical Society \textbf{20} (2007),
  no.~1, 99--148.

\bibitem{delsarte1991spherical}
Philippe Delsarte, Jean-Marie Goethals, and Johan~Jacob Seidel, \emph{Spherical
  codes and designs}, Geometry and Combinatorics, Elsevier, 1991, pp.~68--93.

\bibitem{ding2006r}
Chris Ding, Ding Zhou, Xiaofeng He, and Hongyuan Zha, \emph{R1-{PCA}:
  rotational invariant {L}1-norm principal component analysis for robust
  subspace factorization}, Proceedings of the 23rd International Conference on
  Machine learning, 2006, pp.~281--288.

\bibitem{evans:98}
Lawrence~C. Evans, \emph{{Partial Differential Equations (Graduate Studies in
  Mathematics, V.\ 19) GSM/19}}, American Mathematical Society, June 1998.

\bibitem{hoggar1982t}
Stuart~G Hoggar, \emph{{t-Designs} in projective spaces}, European Journal of
  Combinatorics \textbf{3} (1982), no.~3, 233--254.

\bibitem{katanforoush2003distributing}
Ali Katanforoush and Mehrdad Shahshahani, \emph{Distributing points on the
  sphere, {I}}, Experimental Mathematics \textbf{12} (2003), no.~2, 199--209.

\bibitem{kuijlaars1998asymptotics}
Arno Kuijlaars and E~Saff, \emph{Asymptotics for minimal discrete energy on the
  sphere}, Transactions of the American Mathematical Society \textbf{350}
  (1998), no.~2, 523--538.

\bibitem{kwak2008principal}
Nojun Kwak, \emph{Principal component analysis based on {L}1-norm
  maximization}, IEEE transactions on pattern analysis and machine intelligence
  \textbf{30} (2008), no.~9, 1672--1680.

\bibitem{markopoulos2017efficient}
Panos~P Markopoulos, Sandipan Kundu, Shubham Chamadia, and Dimitris~A Pados,
  \emph{Efficient {L}1-norm principal-component analysis via bit flipping},
  IEEE Transactions on Signal Processing \textbf{65} (2017), no.~16,
  4252--4264.

\bibitem{mccoy2011two}
Michael McCoy and Joel~A Tropp, \emph{Two proposals for robust {PCA} using
  semidefinite programming}, Electronic Journal of Statistics \textbf{5}
  (2011), 1123--1160.

\bibitem{saff1997distributing}
Edward~B Saff and Amo~BJ Kuijlaars, \emph{Distributing many points on a
  sphere}, The mathematical intelligencer \textbf{19} (1997), no.~1, 5--11.

\bibitem{schwartz2013five}
Richard~Evan Schwartz, \emph{The five-electron case of {T}homson’s problem},
  Experimental Mathematics \textbf{22} (2013), no.~2, 157--186.

\bibitem{smale1998mathematical}
Steve Smale, \emph{Mathematical problems for the next century}, The
  mathematical intelligencer \textbf{20} (1998), no.~2, 7--15.

\bibitem{tammes1930origin}
Pieter Merkus~Lambertus Tammes, \emph{On the origin of number and arrangement
  of the places of exit on the surface of pollen-grains}, Recueil des travaux
  botaniques n{\'e}erlandais \textbf{27} (1930), no.~1, 1--84.

\bibitem{thomson1904xxiv}
Joseph~John Thomson, \emph{{XXIV}. on the structure of the atom: an
  investigation of the stability and periods of oscillation of a number of
  corpuscles arranged at equal intervals around the circumference of a circle;
  with application of the results to the theory of atomic structure}, The
  London, Edinburgh, and Dublin Philosophical Magazine and Journal of Science
  \textbf{7} (1904), no.~39, 237--265.

\bibitem{wagner1990means}
Gerold Wagner, \emph{On means of distances on the surface of a sphere (lower
  bounds)}, Pacific Journal of Mathematics \textbf{144} (1990), no.~2,
  389--398.

\bibitem{wagner1992means}
\bysame, \emph{On means of distances on the surface of a sphere. ii.(upper
  bounds)}, Pacific Journal of Mathematics \textbf{154} (1992), no.~2,
  381--396.

\bibitem{yu2020optimal}
Chuanping Yu and Xiaoming Huo, \emph{Optimal projections in the distance-based
  statistical methods}, Statistical Modeling in Biomedical Research, Springer,
  2020, pp.~263--308.

\end{thebibliography}
\end{document}